\documentclass[11pt]{article}

\usepackage{amssymb,amsmath,amsthm}

\usepackage{amsrefs}
\usepackage{cancel}
\usepackage{graphicx}

\usepackage{tikz}
\usetikzlibrary{decorations}
\usetikzlibrary{decorations.pathreplacing}
\usetikzlibrary{decorations.markings}

\BibSpec{arxiv}{%
  +{}{\PrintAuthors} {author}
  +{,}{ \textit} {title}
  +{}{ \parenthesize} {date}
  +{,}{ \eprint} {eprint}
  +{,}{ } {note}
  +{.}{ } {transition}
}

\newcommand{\from}{\colon}

\newcommand{\fp}{\mathop{\scalebox{1.2}{\raisebox{0ex}{$\ast$}}}}

\newcommand{\fpGammaiI}{\fp_i \Gamma_i}

\newcommand{\inters}{\cap}

\renewcommand{\implies}{\Rightarrow}

\newcommand{\thet}{{\{\theta_i\}}}
\newcommand{\ph}{{\{\varphi_i\}}}

\newcommand{\E}{\mathrel{E}}

\newcommand{\concat}{{}^{\smallfrown}}

\renewcommand{\subset}{\subseteq}
\renewcommand{\supset}{\supseteq}

\newcommand{\<}{\langle}
\renewcommand{\>}{\rangle}

\renewcommand{\P}{\mathbb{P}}
 
\newcommand{\comp}[1]{{#1}^{\text{c}}} 

\newcommand{\embeds}{\sqsubseteq}

\newcommand{\restrict}{\restriction}
\newcommand{\cantor}{2^\omega}
\newcommand{\baire}{\omega^\omega}
\newcommand{\define}[1]{\textbf{#1}}

\newcommand{\join}{\oplus}
\newcommand{\bigjoin}{\bigoplus}
\newcommand{\union}{\cup}
\newcommand{\bigunion}{\bigcup}
\newcommand{\disjointunion}{\sqcup}

\newcommand{\biginters}{\bigcap}
\DeclareMathOperator{\dom}{dom}
\DeclareMathOperator{\ran}{ran}

\DeclareMathOperator{\Stab}{Stab}
\DeclareMathOperator{\Free}{Free}

\newcommand{\AD}{\mathrm{AD}}

\newcommand{\ZZ}{\mathrm{Z}}

\newcommand{\DC}{\mathrm{DC}}

\newcommand{\Z}{\mathbb{Z}}
\newcommand{\R}{\mathbb{R}}
\newcommand{\N}{\mathbb{N}}
\newcommand{\F}{\mathbb{F}}
\newcommand{\K}{\mathcal{K}}
\newcommand{\T}{\mathcal{T}}

\theoremstyle{plain}
\newtheorem{thm}{Theorem}[section]
\newtheorem{lemma}[thm]{Lemma}
\newtheorem{prop}[thm]{Proposition}
\newtheorem{cor}[thm]{Corollary}
\newtheorem{question}[thm]{Question}

\newtheorem{conj}[thm]{Conjecture}

\theoremstyle{definition}
\newtheorem{defn}[thm]{Definition}
\newtheorem{remark}[thm]{Remark}
\newtheorem{convention}[thm]{Convention}

\newtheorem*{rep@theorem}{\rep@title}
\newcommand{\newreptheorem}[2]{%
\newenvironment{rep#1}[1]{%
 \def\rep@title{#2 \ref{##1}}%
 \begin{rep@theorem}}%
 {\end{rep@theorem}}}
\makeatother

\newreptheorem{thm}{Theorem}
\newreptheorem{cor}{Corollary}
\newreptheorem{conj}{Conjecture}
\newreptheorem{question}{Question}

\begin{document}

\author{Andrew S.~Marks\thanks{The author was partially supported by the
National Science Foundation under grants DMS-1204907 and DMS-1500974 and
the Turing Centenary research project ``Mind, Mechanism and Mathematics",
funded by the John Templeton Foundation under Award No.\ 15619.}}

\title{Uniformity, Universality, and Computability Theory}

\date{\today}

\maketitle

\begin{abstract}
We prove a number of results motivated by global questions of uniformity in computability theory, and universality of countable Borel equivalence relations. Our main technical tool is a game for constructing functions on free products of countable groups. 

We begin by investigating the notion of uniform universality, first proposed by Montalb\'an, Reimann and Slaman. This notion is a strengthened form of a countable Borel equivalence relation being universal, which we conjecture is equivalent to the usual notion. With this additional uniformity hypothesis, we can answer many questions concerning how countable groups, probability measures, the subset relation, and increasing unions interact with universality. For many natural classes of countable Borel equivalence relations, we can also classify exactly which are uniformly universal. 

We also show the existence of refinements of Martin's ultrafilter on Turing invariant Borel sets to the invariant Borel sets of equivalence relations that are much finer than Turing equivalence. For example, we construct such an ultrafilter for the orbit equivalence relation of the shift action of the free group on countably many generators. These ultrafilters imply a number of structural properties for these equivalence relations. 

\end{abstract}

\pagebreak

\tableofcontents

\pagebreak

\section{Introduction}
\label{sec:intro}
\subsection{Introduction}

In this paper, we investigate a number of problems concerning uniformity
and universality among countable Borel equivalence relations and in
computability theory. 

In a sense, this paper is a sequel to \cite{MR3454384}. In that paper, we used
Borel determinacy to prove a
Ramsey-like theorem (\cite[Lemma 2.1]{MR3454384}) for the shift
action of a free product of two groups. We then gave applications of this result to
several problems in Borel graph combinatorics. In this paper, we generalize
the games used in \cite{MR3454384} to handle shift actions of 
free products of \emph{countably} many groups. The resulting
theorems we prove above these shift actions are the main technical tool of this paper. However, the
applications we give in this paper are to problems in the study of
countable Borel equivalence relations and computability theory, rather than
Borel graph combinatorics. 

We briefly mention two of the results proved using these games. 
Recall that if a group $\Gamma$ acts on the spaces $X$ and
$Y$ then a function $f \from X \to Y$ is said be
\define{$\Gamma$-equivariant} if
$f$ commutes with the actions of $\Gamma$ on $X$ and $Y$, so $\gamma \cdot
f(x) = f(\gamma \cdot x)$ for every $x \in X$ and $\gamma \in \Gamma$.
If $\Gamma$ is a countable discrete group and $X$ is a Polish space, then
$\Gamma$ acts on the Polish space $X^\Gamma$ via the \define{left shift
action}
where for all $x \in X^\Gamma$ and $\gamma, \delta \in \Gamma$ 
\[\gamma \cdot x(\delta) = x(\gamma^{-1} \delta).\] 
One importance of the shift action is
that it is universal in the sense that every Borel action of $\Gamma$ on a Polish space
$X$ admits a Borel embedding 
into the shift action by sending $x \in X$ to the function $\gamma \mapsto
\gamma^{-1} \cdot x$. We will use the notation $\fpGammaiI$ to denote the
free product of the groups $\{\Gamma_i\}_{i \in I}$.

\begin{thm} \label{g1}
  Suppose $I \leq \omega$ and $\{\Gamma_i\}_{i \in I}$ is a set of countable
  discrete groups. Let $\{A_i\}_{i \in I}$
  be a Borel partition of $\left(\cantor \right)^{\fpGammaiI}$. Then there
  exists some $j \in I$ and an injective continuous 
  function $f \from \left(\cantor \right)^{\Gamma_j} \to \left(\cantor
  \right)^{\fpGammaiI}$ that is 
  $\Gamma_j$-equivariant with respect to the shift actions, and 
  such that $\ran(f) \subset A_j$. Furthermore, $f$
  can be chosen to be a Borel reduction between the orbit equivalence
  relations induced by the shift actions of $\Gamma_j$ on $\left(\cantor
  \right)^{\Gamma_j}$ and $\fpGammaiI$ on $\left(\cantor
  \right)^{\fpGammaiI}$.
\end{thm}

That is, if we partition the shift action of the group $\fp_{i \in I}
\Gamma_i$ into $I$ many Borel sets $\{A_i\}_{i \in I}$, not only does one of
the sets $A_j$ contain a copy of the shift action of $\Gamma_j$ (as
witnessed by a continuous equivariant injection), the Borel cardinality of
the shift action of $\Gamma_j$ is also preserved. 

We also have a version of this theorem for the free part of the action. Given a group $\Gamma$ and a Polish space $X$, let 
\[\Free(X^\Gamma) = \{x
\in X^\Gamma : \Stab(x) = \{1\}\}\]
be the free part of the left
shift action: the largest set where every stabilizer under this action is trivial. We have an
analogous theorem for the free parts of these actions:

\begin{thm}\label{g2}
  Suppose $I \leq \omega$ and $\{\Gamma_i\}_{i \in I}$ is a set of countable
  discrete groups. Let $\{A_i\}_{i \in \omega}$ be a Borel partition of
  $\Free(\left( \cantor \right)^{\fpGammaiI})$. Then there exists some $j
  \in I$ and an injective continuous function 
  $f \from \Free((\cantor)^{\Gamma_j}) \to \Free(\left(\cantor
  \right)^{\fpGammaiI})$ that is 
  $\Gamma_j$-equivariant with respect to the shift actions, and
  such that $\ran(f) \subset A_j$. Furthermore, $f$
  can be chosen to be a Borel reduction between the orbit equivalence
  relations induced by the shift actions of $\Gamma_j$ on $\Free((\cantor
  )^{\Gamma_j})$ and $\fpGammaiI$ on $\Free((\cantor)^{\fpGammaiI})$
\end{thm}  

The use of Borel determinacy is necessary to prove Theorems~\ref{g1} and
\ref{g2} as well as the main lemmas from our earlier paper
\cite{MR3454384}.
We show this by proving the following reversal:

\begin{thm}
  Theorems~\ref{g1} and \ref{g2} as well as Lemmas 2.1 and
  3.12 from \cite{MR3454384} are equivalent to Borel determinacy over 
  the base theory $\mathrm{Z}^- + \Sigma_1-\text{replacement} + \DC$.
\end{thm}

In fact, the reversal of all these theorems requires only one of their
simplest nontrivial cases: when there are two groups in the free product which are
both copies of $\Z$. It remains an open question whether any of the
consequence of these theorems--both in the current paper and in 
\cite{MR3454384}--require the use of Borel determinacy in their
proofs.\footnote{Sherwood Hachtman has resolved these questions. See the
footnote before Question~\ref{rq1}.}

In Section~\ref{sec:ri}, we turn to questions of uniformity in the study of
universal countable Borel equivalence relations. We work here in the
setting where a countable Borel equivalence relation $E$ on a standard
Borel space $X$ is given
together with some fixed family of partial Borel functions $\ph_{i \in
\omega}$ on $X$ that
generate $E$, for which we use the notation $E_{\ph}$. In this context, we
say that a homomorphism $f$ between $E_{\ph}$ and $E_{\thet}$ is
uniform if a witness
that $x E_{\ph} y$ (that is, a pair of indices $(i,j)$ so that $\varphi_i(x) =
y$ and $\varphi_j(y) = x$) can be
transformed into a witness that $f(x) E_{\thet} f(y)$ in a way that is
independent of $x$ and $y$ (see Section~\ref{subsec:intro_to_uu} for a
precise definition). If these
equivalence relations $E_{\ph}$ and $E_{\thet}$ are generated by free actions of countable groups,
this is equivalent to the assertion that the \define{cocycle} (see \cite[Appendix B]{MR2155451}) associated to $f$ is
group homomorphism. 

This sort of uniformity arises in the study of Martin's conjecture, where
Slaman and Steel have shown that Martin's conjecture for Borel functions is
equivalent to the statement that every homomorphism from Turing equivalence
to itself is equivalent to a uniform homomorphism on a Turing
cone~\cite{MR960895}. 
This idea also arises often in proving
nonreducibility results between countable Borel equivalence relations;
first we analyze the class of homomorphisms which are uniform on some
``large'' (e.g. conull) set, and then prove that
this analysis is complete by showing every homomorphism is equivalent to a
uniform homomorphism on a large set.
For instance, this is a typical proof strategy in applications of
cocycle superrigidity to the field of countable Borel equivalence
relations. See for example~\cites{MR776417, MR1775739, MR2155451, MR1903855,
MR2500091}.

One of the central concepts we study is the idea of uniform universality
for countable Borel equivalence relations, which was introduced in
unpublished work by 
Montalb\'an, Reimann and Slaman. 
Precisely, $E_{\ph}$ is said to be uniformly universal (with respect to
$\ph_{i \in \omega}$) if there is a uniform reduction of every countable
Borel equivalence relation, presented as $E_{\thet}$, to $E_{\ph}$.
Certainly, 
all known universal
countable Borel equivalence relations are uniformly universal with respect
to the way they are usually generated, so it is fair to say that the 
uniformly universal countable
Borel equivalence relations include all countable Borel equivalence relations that we
can hope to prove universal without 
dramatically new techniques. We make the following
stronger conjecture (see Conjecture~\ref{conj:uniformly_universal}):

\begin{conj}\label{intro_uu_conj}
  A countable Borel equivalence relation is universal if and only if it is
  uniformly universal with respect to every way it can be generated.
\end{conj}

One attraction of the notion of uniform universality is that we are able to 
settle many open questions about universality with this additional
uniformity assumption,
and we can also prove precise characterizations of what equivalence
relations are uniformly universal in many settings. 
The bulk of Section~\ref{sec:ri} is devoted to theorems of these sorts, and
we outline some of our main results:

\begin{thm}[Properties of uniform universality]\label{uu_properties}
\mbox{ }
\begin{enumerate}
  \item For every countable group $\Gamma$, there is a Borel action of $\Gamma$
  generating a uniformly universal countable Borel equivalence relation if
  and only if $\Gamma$ contains a copy of $\F_2$, the free group on two
  generators. 

  \item Given any uniformly universal countable Borel equivalence relation
  $E_{\ph}$ on a standard probability space $(X,\mu)$, there is a
  $\mu$-conull set $A$ for which $E_{\ph} \restrict A$ is not uniformly
  universal. 

  \item Both the uniformly universal and non-uniformly universal countable Borel equivalence relations
  are cofinal under $\subset$.

  \item An increasing union of non-uniformly universal countable Borel
  equivalence relations is not uniformly universal. An increasing union of
  uniformly universal countable Borel equivalence relations need not be
  uniformly universal.
\end{enumerate}
\end{thm}

\begin{thm}[Classifications]
\label{uu_classifications}
\mbox{ }
\begin{enumerate}
  \item If $\Gamma$ is a countable group, then the shift action of $\Gamma$
  on $2^\Gamma$ generates a uniformly universal countable Borel equivalence
  relation if and only if $\Gamma$ contains a subgroup isomorphic to
  $\F_2$.
  
  \item If $\Gamma$ is a countable group, then the 
  conjugacy action of $\Gamma$ on its subgroups is uniformly universal
  if and only if $\Gamma$ contains a subgroup isomorphic to $\F_2$. 

  \item (Joint with Jay Williams) If $G$ is a countable subgroup of
  $S_\infty$ and $X$ is a standard Borel space of cardinality at least $3$,
  then the permutation action of $G$ on $X^\omega$ is uniformly
  universal if and only if there exists some $n \in \omega$ and a subgroup
  $H \leq G$ isomorphic to $\F_2$ such that the map $H \to
  \omega$ given by $h \mapsto h(n)$ is injective.

  \item For every additively indecomposable $\alpha < \omega_1$, define the 
  equivalence relation $\equiv_{(<\alpha)}$ on $\cantor$ by $x
  \equiv_{(<\alpha)} y$ if there exists $\beta < \alpha$ such that
  $x^{(\beta)} \geq_T y$ and $y^{(\beta)} \geq_T x$. Then $\equiv_{(<\alpha)}$
  is uniformly universal if and only if there is a
  $\beta < \alpha$ such that $\beta \cdot \omega = \alpha$. 

  \item If $E_{\ph_{i \in \omega}}$ is a countable Borel equivalence relation on
  $\cantor$ coarser than recursive isomorphism and closed under countable
  uniform joins, then $E_{\ph_{i \in \omega}}$ is not uniformly
  universal. (This includes many-one equivalence, tt equivalence, wtt
  equivalence, Turing equivalence, enumeration equivalence, etc.).
  Further, this result is not true when $2^\omega$ is replaced by
  $3^\omega$.
\end{enumerate}
\end{thm}

Note that several of these results settle open questions about universal
equivalence relations in our more restrictive uniform context. 
For example, Hjorth~\cite{MR1815088}*{Question 1.4}
\cite{MR1900547}*{Question 6.5.(A)} has asked if 
$E$ and $F$ are 
countable Borel equivalence relations, $E$ is universal, and $E \subset F$,
must $F$ be universal? This is addressed by part (3) of
Theorem~\ref{uu_properties}. Thomas~\cite{MR2500091}*{Question 3.22} has
asked whether every universal countable Borel equivalence relation on a
standard probability space is non-universal on a conull set. This is
addressed by part (2) of Theorem~\ref{uu_properties}. Finally, Thomas has
conjectured~\cite{MR2914864}*{Conjecture 1.5} that free Burnside groups of sufficiently high
rank admit Borel actions generating universal countable Borel equivalence
relations. This contradicts part (1) of Theorem~\ref{uu_properties}, when
combined with Conjecture~\ref{intro_uu_conj}.

Many of these abstract properties of uniformly universal equivalence
relations
hinge upon an analysis of equivalence relations
from computability theory (some of which are mentioned in the second theorem we
have stated above).
Of course, the investigation of the universality
of equivalence relations from computability theory is interesting in its own
right. For example, whether Turing equivalence is a universal countable Borel
equivalence relation is a long open question of Kechris~\cite{MR1233813},
and is closely connected with Martin's conjecture on Turing invariant
functions, as discussed in~\cite{1109.1875} and~\cite{MR1770736}. 

Another important example of a computability-theoretic equivalence relation
is that of recursive isomorphism.
Suppose $Z$ is a countable set and $G$ is a group of permutations of $Z$. Then
the \define{permutation action} of $G$ on $Y^Z$ is defined by $(g \cdot y)(z) =
y(g^{-1}(z))$ for $g \in G$, $y \in Y^Z$, and $z \in Z$. Now the
equivalence relation of \define{recursive isomorphism} on $Y^{\omega}$ is
defined to be the orbit equivalence relation of the permutation action of
the group of computable bijections of $\omega$ on $Y^\omega$.
The universality of recursive isomorphism on $\cantor$ is a long open
question. However, several partial results are known: Dougherty
and Kechris have shown that recursive isomorphism on $\baire$ is
universal~\cites{DoughertyKechris,MR1770736}, which was later improved to
$5^\omega$ by Andretta, Camerlo, and Hjorth~\cite{MR1815088}. 
We further improve this result to $3^\omega$:

\begin{thm}
  Recursive isomorphism on $3^\omega$ is a universal countable Borel
  equivalence relation.
\end{thm}

This argument hinges on reducing the problem of universality to a
combinatorial problem involving Borel colorings of a family of $2$-regular
Borel graphs. The problem can be solved if we are working on $3^\omega$,
however on $2^\omega$, the problem can not be solved, as we demonstrate
using our new game-theoretic tools (see Theorem~\ref{countable_3_coloring}). 
Indeed, the games of this paper and of~\cite{MR3454384} were originally developed specifically to solve this
problem.
We also restate
this combinatorial problem in computability-theoretic language, showing that it
is equivalent to an inability to control the computational power of countable uniform joins
in a seemingly simple context (see Lemma~\ref{cor:diagonalization}). These
results are the basis of many of our results mentioned above giving
families of equivalence relations which are not uniformly universal.

Now while this combinatorial problem related to the universality of
recursive isomorphism on $\cantor$ can not be solved in general, it turns
out that it can be solved modulo a nullset with respect to any Borel
probability measure. 
This leads to the following theorem, and part
(2) of Theorem~\ref{uu_properties}. Recall that a countable Borel
equivalence relation $E$ is said to be measure universal if given any Borel
equivalence relation $F$ on a standard probability space $(X,\mu)$, there
is a Borel $\mu$-conull set $A$ such that $F \restrict A$ can be Borel
reduced to $E$.

\begin{thm}
  Recursive isomorphism on $2^\omega$ is a measure universal countable
  Borel equivalence relation.
\end{thm}

Hence, measure-theoretic tools can not be used to show that recursive
isomorphism on $\cantor$ is not universal.

Though it remains open whether
recursive isomorphism is uniformly universal, we can rule out a large class
of uniform proofs based on the close connection between recursive
isomorphism and many-one equivalence. Thus, we conjecture the following: 

\begin{conj}\label{ri_not_universal}
  Recursive isomorphism on $2^\omega$ is not a universal countable Borel
  equivalence relation.   
\end{conj}

If this conjecture is true, it implies the existence
of a pair of Borel equivalence relations of different Borel cardinalities
($E_\infty$ and recursive isomorphism), for which we can not prove this
fact using measure-theoretic tools. Conjecture~\ref{ri_not_universal} also
implies that Turing equivalence is not universal, since it is Borel
reducible to recursive isomorphism via the Turing jump.

One of the motivations of this paper is the search for new tools with which to
study countable Borel equivalence relations. Currently, measure theoretic
methods are the only known way of proving nonreducibility among countable
Borel equivalence relations of complexity greater than $E_0$ (see
Question~\ref{reducible_vs_measure_reducible}). However, many important
open problems in the subject are known to be resistant to measure-theoretic
arguments. For example, it is open whether an increasing
union of hyperfinite Borel equivalence relations is
hyperfinite~\cite{MR1149121}, and whether recursive isomorphism on
$2^\omega$ is a universal countable Borel equivalence
relation~\cite{MR1770736}. However, an increasing union of hyperfinite
Borel equivalence relations is hyperfinite modulo a nullset with respect to
any Borel probability measure by a theorem of Dye and Kreiger~\cites{MR0131516, MR0158048,
MR0240279}, and every countable Borel equivalence relation can be embedded
into recursive isomorphism on $\cantor$ modulo a nullset by
Theorem~\ref{universality_of_ri} of this paper. 

One promising candidate for such a new, non-measure-theoretic tool is
Martin's ultrafilter on the Turing-invariant sets. Martin's
ultrafilter is sometimes called Martin measure. However, because we
will discuss Borel probability measures often in this paper, we will use
the terminology of ultrafilters to keep this distinction clear. Martin has
conjectured a complete classification of the homomorphisms from Turing
equivalence to itself with respect to this ultrafilter which would have
many consequences for the study of countable Borel equivalence relations. See~\cite{1109.1875} for a survey of
connections between Martin's conjecture and countable Borel equivalence
relations. Many of the results discussed there are due to Simon Thomas, who
first recognized the tremendous variety of consequences of Martin's
conjecture for the field of countable Borel equivalence relations, beyond
just the non-universality of Turing equivalence.

In Section~\ref{sec:ultrafilters}, we prove several
structure theorems for certain countable Borel equivalence relations using
ultrafilters related to Martin's ultrafilter, but which are defined on the
quotient of equivalence relations finer than Turing equivalence. 
To begin, 
we generalize results from~\cite{1109.1875}, where we showed
that $E_\infty$ is not a smooth disjoint union of Borel equivalence
relations of smaller Borel cardinality and that
$E_\infty$ achieves its universality on a nullset with respect to any Borel
probability measure (which answered questions of
Thomas~\cite{MR2500091}*{Question 3.20} and Jackson, Kechris, and
Louveau~\cite{MR1900547}*{Question 6.5.(C)}). We show that these results
are all true with $E_\infty$ replaced with a larger class of universal
structurable Borel equivalence relations which also includes, for example,
the universal treeable countable Borel equivalence relation, and the
equivalence relation of isomorphism of contractible simplicial complexes.

Our proof of these results uses our games to define a $\sigma$-complete
ultrafilter $U$ on the $\sigma$-algebra of Borel $E$-invariant sets for
each of these equivalence relations $E$. These ultrafilters are very
closely related to Martin's ultrafilter; the equivalence relations $E$ we
consider are all subsets of Turing equivalence, and our new ultrafilters
agree with Martin's ultrafilter when restricted to Turing invariant sets.
Further, these ultrafilters have structure-preserving properties
reminiscent of Martin's ultrafilter; if $A \in U$, then $E \leq_B E
\restrict A$ (i.e. $U$ preserves the Borel cardinality of $E$). 

We briefly describe the class of equivalence relations for which we obtain
these ultrafilters. 
Suppose $\K$ is a
Borel class of countable structures closed under isomorphism. Then a
countable Borel equivalence relation $E$ is said to be $\K$-structurable if
there is a Borel way of assigning a structure from $\K$ to every $E$-class
whose universe is that $E$-class. This notion was defined by Jackson,
Kechris and Louveau in~\cite{MR1900547}, and Ben Miller has pointed out
that their ideas can be used to show that 
for every such $\K$, there
is a universal $\K$-structurable countable Borel equivalence relation
$E_{\infty \K}$ (see Theorem~\ref{universal_structurable}). 
It is these equivalence relations for which we obtain our ultrafilters,
under the assumption of one more condition: that the class
 of
$\K$-structurable equivalence relations is closed under independent joins
(see Section~\ref{subsec:k_structurable} for a definition).
We now state our result precisely (see Theorems~\ref{universal_on_null},
\ref{K_ultrafilters}, and \ref{K_embedding}). 

\begin{thm}\label{KS_properties}
  Suppose $\K$ is a Borel class of countable structures closed under
  isomorphism, and let $E_{\infty \K}$ be the universal $\K$-structurable
  equivalence relation on the space $Y_{\infty \K}$. Then
  \begin{enumerate}
  \item If $\mu$ is a Borel probability measure on $Y_{\infty \K}$, there
  is a $\mu$-null Borel set $A$ so that $E_{\infty \K} \embeds^i_B
  E_{\infty \K} \restriction A$.
  \item If 
  the class of $\mathcal{K}$-structurable countable Borel
  equivalence relations is closed under binary independent joins, then there is a
  Borel cardinality preserving ultrafilter on the quotient space of
  $E_{\infty \K}$ and hence $E_{\infty \K}$ is not a smooth disjoint union of
  equivalence relations of smaller Borel cardinality. 
  \item If the class of $\mathcal{K}$-structurable countable Borel
  equivalence relations is closed under countable independent joins, then
  if $\{A_i\}_{i \in \omega}$ is a Borel partition of $E_{\infty \K}$ into
  countably many sets, then there exists some $A_i$ such that $E_{\infty
  \K} \embeds_B E_{\infty \K} \restriction A_i$. It follows that  
  for all
  countable Borel equivalence relations $F$, $E_{\infty K} \leq_B F$
  implies $E_{\infty K} \embeds_B F$.
  \end{enumerate}
\end{thm}

In a future paper joint with Adam Day, we use the ultrafilter constructed
here for the universal treeable countable Borel equivalence relation to
prove a strengthening of Slaman and Steel's result from \cite{MR960895}
that Martin's conjecture for Borel functions is equivalent to every Turing
invariant function being uniform Turing invariant on a pointed perfect set.

We hope that these
ultrafilters will continue to be useful tools for studying countable Borel
equivalence relations in the future. Because they preserve
Borel cardinality, they have the potential for proving much sharper
theorems than other tools which do not necessarily have this property, such
as Borel probability measures. 
Finally, we remark that equivalence
relations of the form $E_{\infty \K}$ appear to be underappreciated as
natural examples of countable Borel equivalence relations, 
and many interesting open questions exist regarding how
model-theoretic properties of structures in the class $\K$ influence the
complexity of the resulting $E_{\infty \K}$\footnote{Since a first draft
of this paper was circulated, these questions have been 
investigated by Chen and Kechris \cite{ChenKechris}}.

\subsection{Basic definitions, notation, and conventions}
\label{subsec:notation}

Throughout we will use $X$, $Y$, and $Z$ for standard Borel spaces, $x$,
$y$, and $z$ for
elements of these spaces, and $A$, $B$, and $C$ for subsets of standard Borel spaces
(which will generally be Borel). We will use $f$, $g$, and $h$ for functions
between standard Borel spaces. If $A$ is a subset of a standard Borel
space, we will use $\comp{A}$ to denote its complement.

A \define{Borel equivalence relation} on a standard Borel space $X$ is an
equivalence relation on $X$ that is Borel as a subset of $X \times X$. We
will generally use $E$ and $F$ to denote Borel equivalence relations. If
$E$ and $F$ are Borel equivalence relations on the standard Borel spaces
$X$ and $Y$, then $f \from X \to Y$ is said to be a \define{homomorphism} from
$E$ to $F$ if for all $x, y \in X$, we have $x E y \implies f(x) F f(y)$.
We say that $E$ is \define{Borel reducible}
to $F$, noted $E \leq_B F$, if there is a Borel function $f \from X \to Y$
so that for all $x, y \in X$, we have $x E y \iff f(x) F f(y)$. 
Such a function
induces an injection $\hat{f} \from X/E \to Y/F$. The class of Borel equivalence
relations under $\leq_B$ has a rich structure that been a major topic of
research in descriptive set theory in the past few decades. The field
has had remarkable success both in calibrating the difficulty of
classification problems of interest to working mathematicians, and also in
understanding the abstract structure of the space of all classification problems.

If $E$ is a Borel equivalence relation on the standard Borel
space $X$, then $A \subset X$ is said to be \define{$E$-invariant} if $x \in A$ and
$x E y$ implies $y \in A$. If a group $\Gamma$ acts on a space $X$, then we
say that $A \subset X$ is \define{$\Gamma$-invariant} if it is invariant
under the orbit equivalence relation of the $\Gamma$ action.
If $P$ is a Borel property of elements of $x$,
then we will often consider the largest $E$-invariant subset of $X$
possessing this property. Precisely, this is the set of $x$ such that for
all $y \in X$ where $y E x$, $y$ has property $P$. 

If $E$ and $F$ are equivalence relations on standard Borel spaces $X$ and
$Y$, then a \define{Borel embedding} of $E$ into $F$ is an injective Borel reduction
from $E$ to $F$. If there is a Borel embedding from $E$ to $F$ we denote
this by $E \embeds_B F$. An \define{invariant Borel
embedding} is one whose range is $F$-invariant. If there is an invariant
Borel embedding from $E$ to $F$ we denote this by $E \embeds_B^i F$. 

A Borel equivalence relation is said to be \define{countable} if all of its
equivalence classes are countable. 
A countable Borel equivalence relation
$E$ is said to be \define{universal} if for all countable Borel equivalence
relations $F$, we have $F \leq_B E$. Universal countable Borel equivalence
relations arise naturally in many areas of mathematics. For example,
isomorphism of finitely generated groups~\cite{MR1700491}, conformal equivalence of Riemann
surfaces~\cite{MR1731384}, and isomorphism of locally finite connected
graphs~\cite{MR1791302} are all
universal countable Borel equivalence relations.

If $E$ is a countable Borel equivalence relation on $X$, we will use $\phi,
\psi, \theta$ to denote partial Borel functions $X \to X$ whose graphs are
contained in $E$. By a theorem of Feldman and Moore~\cite{MR0578656}, for
every countable Borel equivalence relation $E$, there exists 
countably many Borel involutions $\{\phi_i\}_{i \in \omega}$ of $X$ such that $x E y$ if
and only if there exists an $i$ such that $\phi_i(x) = y$. Hence, every
countable Borel equivalence relation is generated by the Borel action of
some countable group. 

Throughout, we use $\Gamma$ and $\Delta$ to denote countable groups, which
we always assume to be discrete, and we use the lowercase $\alpha, \beta, \gamma,
\delta$ for their elements. 
If $X$ is a standard Borel space, then so is
the space $X^{\Gamma}$ of functions from $\Gamma$ to $X$ whose standard
Borel structure arises from the product topology. 
We let $E(\Gamma,X)$ denote the equivalence relation on
$X^{\Gamma}$ of orbits of the left shift action where $x E(\Gamma,X) y$ if
there is a $\gamma \in \Gamma$ such that $\gamma \cdot x = y$. By
\cite{MR1149121}, if $\Gamma$ contains a subgroup isomorphic
to the free group $\F_2$ on two generators, 
and $X$ has cardinality $\geq 2$, then $E(\Gamma,X)$ is a universal countable Borel
equivalence relation. 

If a group $\Gamma$ acts on a space $X$, then the \define{free part} of this action
is the set $Y$ of $x \in X$ such that for every nonidentity $\gamma \in
\Gamma$, we have $\gamma \cdot x \neq x$. We use the notation
$\Free(X^{\Gamma})$ to denote the free part of the left shift action of
$\Gamma$ on $X^\Gamma$. We will also let $F(\Gamma,X)$ denote the
restriction of the equivalence relation $E(\Gamma,X)$ to this free part. 

A \define{Borel graph} on a standard Borel space $X$ is a symmetric irreflexive
relation on $X$ that is Borel as a subset of $X \times X$. If $\Gamma$ is a
marked countable group (i.e. a group equipped with a generating set) and $X$ is a standard Borel space, then we let
$G(\Gamma, X)$ note the graph on $\Free(X^\Gamma)$ where there is an edge
between $x$ and $y$ if there is a generator $\gamma$ of $\Gamma$ such that
$\gamma \cdot x = y$ or $\gamma \cdot y = x$. Hence, the connected
components of $G(\Gamma,Y)$ are the equivalence classes of $F(\Gamma,X)$.

A \define{Borel $n$-coloring} of a Borel
graph $G$ is a function $f$ from the vertices of $G$ to $n$ such that if
$x$ and $y$ are adjacent vertices, then $f(x) \neq f(y)$. A fact we use
several times is that $G(\Z,2)$ has no Borel $2$-coloring, equipping the
additive group $\Z$ with its usual set of $\{1\}$
(see~\cite{MR1667145}). A graph is said to be \define{$d$-regular} if every vertex
of the graph has exactly $d$ neighbors. We also have from~\cite{MR1667145}
that every Borel $d$-regular graph has a Borel $(d+1)$-coloring.

Give two reals $x, y \in \cantor$, the \define{join} of $x$ and $y$, noted $x \join
y$ is defined by setting $(x \join y)(2n) = x(n)$ and $(x \join y)(2n+1) =
y(n)$ for all $n$. The join of finitely many reals is defined analogously.
If we fix some computable bijection $\<\cdot,\cdot\>
\from \omega^2 \to \omega$, we define the uniform computable
join $\bigjoin_{i \in \omega} x_i$ of countably many reals $x_0, x_1,
\ldots$ by setting $\bigjoin_{i \in \omega} x_i (\<n,m\>) = x_n(m)$.
If $s \in 2^{< \omega}$ and $x \in \cantor$, we use $s \concat x$ to denote
the concatenation of $s$ followed by $x$.

We use $x'$ to denote the Turing jump of a real $x$. If $\alpha$ is a
notation for a computable ordinal, then we let $x^{(\alpha)}$ denote the 
$\alpha$th iterate of the Turing jump relative to $x$. 

\subsection{Acknowledgments}

Theorems~\ref{some_family} and \ref{universality_of_ri} are from the
author's thesis \cite{MarksPhD}. The author would like his thesis advisor,
Ted Slaman, for many years of wise advice.

The author would also like to thank Clinton Conley, Adam Day, Alekos
Kechris, Ben Miller, Jan Reimann, Richard Shore, John Steel, Simon Thomas,
Anush Tserunyan, Robin Tucker-Drob, Jay Williams, Hugh Woodin, and Jind\v{r}ich
Zapletal for many helpful conversations. 

\pagebreak

\section{Games and equivariant functions} \label{sec:main_idea}

\subsection{The main game}
\label{subsec:games_intro}

In this section we introduce the games which are the main technical tool of
this paper. They are the natural
generalization of the games in \cite{MR3454384} to free products of
countably many groups, and we use many of the same ideas as that paper.
Throughout this section, we fix $I \leq \omega$ and 
some countable collection $\{\Gamma_i\}_{i \in I}$ of disjoint countable
groups. For each $i \in I$, we also fix a listing $\gamma_{i,0}, \gamma_{i,1}
\ldots$ of the nonidentity elements of $\Gamma_i$. 
We will often abbreviate our indexing for clarity. For example, we write
$\left(\prod_i \cantor
\right)^{\fpGammaiI}$ instead of $\left(\prod_{i \in I} \cantor
\right)^{\fp_{i \in I} \Gamma_i}$.

Of course, a countable product of copies of $\cantor$ is homeomorphic to
$\cantor$. However, throughout this section we will work with the space $\prod_i
\cantor$ instead of $\cantor$ to streamline the notation in some of our proofs. We also
will not use any particular properties of $\cantor$ in this section, which
could be replaced by $\baire$ or even the space $2$ (that is, $\prod_i
\cantor$ would become $\prod_i 2$). We use the space
$\cantor$ since it will be convenient in Section~\ref{sec:ultrafilters}. 

We begin with a definition we use throughout in order to partition $\fp_i \Gamma_i$:

\begin{defn}
A group element $\alpha \in \fpGammaiI$ is called a  
\define{$\Gamma_j$-word} if $\alpha$ is not the identity and it begins with an element of
$\Gamma_j$ as a reduced word.
\end{defn}

Thus, the group $\fp_i \Gamma_i$ is the disjoint union of the set
containing the identity $\{1\}$, and the set of $\Gamma_j$-words for
each $j \in I$.

Fix any $j \in I$. 
We will be considering games
for building an element $y \in \left(\prod_i \cantor \right)^{\fpGammaiI}$
where player I defines $y$ on $\Gamma_j$-words, player II
defines $y$ on all other nonidentity group elements, and both
players contribute to defining $y$ on the identity. 
We begin by giving a definition that we use to organize the turns on
which the bits of $y(\alpha)$ are defined. 

\begin{defn}\label{turn_defn}
  We define the \define{turn
  function}
  $t \from \fpGammaiI \to \omega$ as follows. First we set $t(1) = 0$. Then, for each
  nonidentity element $\alpha$ of the free product $\fpGammaiI$, there is a
  unique sequence $(i_0,k_0), \ldots, (i_n,k_n)$ such that $i_m \neq i_{m+1}$
  for all $m$ and $\alpha = \gamma_{i_0,k_0} \ldots
  \gamma_{i_n,k_n}$. We define $t(\alpha)$ to be the 
  least $l$ such that  
  $i_m + m < l$ and $k_m +m < l$ for all $m \leq n$. 
\end{defn}

The key property of this definition is that if $i,k \leq l+1$ and $\alpha$
is not a $\Gamma_i$-word, then $t(\gamma_{i,k} \alpha) \leq l+1$ if and
only if $t(\alpha) \leq l$. We also have that for each $l$ there are only
finitely many $\alpha$ with $t(\alpha) = l$. 

We will also define a partition $\{W_i\}_{i \in I}$ of the set $\fpGammaiI
\times I$. 
\begin{defn}
  For each $j \in I$, let 
  $W_j$ be the set of $(\alpha,i) \in \fpGammaiI \times I$
  such that $\alpha$ is a $\Gamma_j$-word, or $\alpha = 1$ and $i = j$.
\end{defn}

We are now ready to define our main game:

\begin{defn}[The main game]\label{game_defn}
  Fix a bijection $\< \cdot, \cdot\> \from I \times \omega \to \omega$. 
  Given any $A \subset \left(
  \prod_{i} \cantor \right)^{\fpGammaiI}$, and any $j \in I$, we define
  the following game $G^A_j$ for producing $y \in \left(
  \prod_{i} \cantor \right)^{\fpGammaiI}$. 
  Player I goes
  first, and on each turn the players alternate defining $y(\alpha)(i)(n)$ for finitely
  many triples $(\alpha,i,n)$. 
  Player I will determine $y(\alpha)(i)(n)$ if
  $(\alpha,i) \in W_j$ and Player II will determine $y(\alpha)(i)(n)$
  otherwise. Finally, the value $y(\alpha)(i)(n)$ will be defined on turn $k$ of the
  game by the appropriate player 
  if $t(\alpha) + \<i,n\> = k$. Player I wins the game if and only if the $y$ that is
  produced is not in $A$.
\end{defn}

Note that $y(\alpha)(i)$ is an element of
$\cantor$ and so $y(\alpha)(i)(n)$ is its $n$th bit. Note also that 
$t(\alpha)$ is the first turn on
which $y(\alpha)(i)(n)$ is defined for some $i$ and $n$, and on turn
$t(\alpha) + l$, we have that $y(\alpha)(i)(n)$ is defined for the $l$th
pair $(i,n)$. 

Now we prove two key lemmas. The first concerns strategies for player
I:

\begin{lemma}\label{player_I_lemma}
  Suppose $\{A_i\}_{i \in I}$ is a partition of $\left(\prod_{i} \cantor
  \right)^{\fpGammaiI}$. Then player I can not have a winning strategy in
  $G^{A_i}_i$ for every $i \in I$.
\end{lemma}
\begin{proof}
  We proceed by contradiction. We claim that from such winning strategies,
  we could produce a $y$ that was
  simultaneously a winning outcome of player I's strategy in $G^{A_i}_i$
  for every
  $i \in I$, and hence $y \notin A_i$ for all $i \in I$, contradicting the fact
  that $\{A_i\}_{i \in I}$ is a partition.

  To see this, fix winning strategies for player I in each $G^{A_i}_i$.  
  Inductively, assume all turns $< k$ of the games $G^{A_i}_i$ have been played and
  that we have already defined $y(\alpha)(i)(n)$ for all $(\alpha,i,n)$
  where $t(\alpha) + \<i,n\> < k$. Now on turn $k$ of the game, the winning
  strategies for player I in the games $G^{A_i}_i$ collectively define
  $y(\alpha)(i)(n)$ on all $(\alpha,i,n)$ where $t(\alpha) + \<i,n\> = k$.
  Since the sets $W_j$ partition $\fpGammaiI \times I$, there is no
  inconsistency between any of these moves in different games. This defines 
  $y(\alpha)(i)(n)$ for all $(\alpha,i,n)$ where $t(\alpha) + \<i,n\> < k+1$.
  We can now move
  for player II in each of the games $G^{A_i}_i$ using this information,
  finishing the $k$th turn of all these games. This completes the
  induction.
\end{proof}

There is a different way of viewing Lemma~\ref{player_I_lemma} which the
reader may find helpful. One can regard the games $G^{A_i}_i$ in
Definition~\ref{game_defn} as constituting a single game with countably
many players: one player for each $i \in I$. The role of player $i$ in this
multiplayer game corresponds to the role of player $I$ in the game
$G^{A_i}_i$. That is, in this multiplayer game we are still building an
element $y \in \left( \prod_{i} \cantor \right)^{\fpGammaiI}$, but now
player $i$ defines $y(\alpha)(i)$ for all $(\alpha,i) \in W_i$. The proof
of Lemma~\ref{game_defn} is essentially checking that the games $G^{A_i}_i$
fit together in this way. 

Now in this multiplayer game, instead of declaring a winner, one of the
players is instead declared the loser once play is over. That is, the
payoff set for the multiplayer game is a Borel partition $\{A_i\}_{i \in I}$ of
$\left( \prod_{i} \cantor \right)^{\fpGammaiI}$, and player $i$ loses if
the $y$ that is created during the game is in $A_i$. It is a trivial
consequence of Borel determinacy that in such a multiplayer game, there
must be some player $i$ so that the remaining players have a strategy to
collaborate to make player $i$ lose. (Otherwise, if every player has a
strategy to avoid losing, playing these strategies simultaneously yields an
outcome of the game not in any element of the partition).

We will make use of following projections from $(\prod_i \cantor)^{\fpGammaiI}$
to $(\cantor)^{\Gamma_i}$.

\begin{defn}\label{easy_proj}
  For each $i \in I$, let $\pi_i \from (\prod_i \cantor)^{\fpGammaiI} \to
  (\cantor)^{\Gamma_i}$ be the function $\pi_i(x)(\gamma) = x(\gamma)(i)$,
\end{defn}

  Note that $\pi_i$ is $\Gamma_i$-equivariant.

Our second key lemma is a way of combining strategies for player II
in the game $G_i$. In the alternate way of viewing things described
above, since there must be some player $i$ so that the remaining players
(which viewed jointly are player II in the game $G_i$) can collaborate to
make player $i$ lose, we are now interested in what can be deduced from the
existence of such a strategy. 

Note that below we speak of a strategy in the game $G_j$ instead of $G_j^A$
(suppressing the superscript) to emphasize that this lemma does not
consider a particular payoff set. Eventually in Lemma~\ref{ultrafilter_lemma} we
will apply this lemma when $s_\gamma$ depends on $\gamma$.

\begin{lemma}\label{player_II_lemma}
  Fix a $j \in I$, and suppose that to each element $\gamma$ of $\Gamma_j$ we
  associate a strategy $s_\gamma$ for player II in the game $G_j$. Then
  there is a $y \in \left(\prod_i\cantor\right)^{\fpGammaiI}$ such that for
  all $\gamma \in \Gamma_j$, $\gamma \cdot y$ is an outcome of
  the game $G_j$ where player II uses the strategy 
  $s_\gamma$. Further, for 
  every $z \in (\cantor)^{\Gamma_j}$, there is a unique such $y$ so that 
  $\pi_j(y) = z$.
\end{lemma}
\begin{proof}
  Fix a $z \in (\cantor)^{\Gamma_j}$. For each $\gamma \in
  \Gamma_j$ we will play an instance of the game $G_j$ whose outcome will
  be $\gamma \cdot y$, where the moves for player II are made by the
  strategy $s_\gamma$. We will specify how to move
  for player I in these games. Indeed, at each turn, there will be a unique
  move for player I that will satisfy our conditions above.
  We play these games for all $\gamma \in \Gamma_j$
  simultaneously. 
  
  If $\gamma \in \Gamma_j$ and so $\gamma \cdot y$ is an outcome of a play of $G_j$ where player II uses
  the strategy $s_\gamma$, then this strategy determines $(\gamma \cdot
  y)(\alpha)(i) = y(\gamma^{-1} \alpha)(i)$ for $(\alpha,i) \notin W_j$. So
  let $V_{\gamma} = \{(\gamma^{-1} \alpha,i): (\alpha,i) \notin W_j\}$ so
  that $V_{\gamma}$ is the set of $(\beta,i)$
  such that player II determines $y(\beta)(i)$ when they move in the game
  associated to $\gamma$ whose outcome
  is $\gamma \cdot y$. Note that the $V_\gamma$ are pairwise disjoint.
  This is because if
  $\alpha$ is not a $\Gamma_j$-word and $\gamma \in \Gamma_j$, then 
  $\gamma^{-1} \alpha$ is a reduced word, and so both $\gamma^{-1}$ and
  $\alpha$ can be uniquely determined from $\gamma^{-1} \alpha$.
  Hence the strategies for player II in these different games do not
  interfere with each other. 
  
  Inducting on $k$, suppose $(\gamma \cdot y)(\alpha)(i)(n)$ is defined for
  all $\gamma \in \Gamma_j$, $\alpha \in \fpGammaiI$, and $i, n \in \omega$
  such that $t(\alpha) + \<i,n\> < k$, and all moves on turns $< k$ have been
  played in the games. Suppose $\gamma \in \Gamma_j$. We will begin by
  making the $k$th move for player I in the game defining $\gamma \cdot y$.
  First, if $k = \<j,n\>$ for some $n$, we must define $(\gamma \cdot
  y)(1)(j)(n) = y(\gamma^{-1})(j)(n) = z(\gamma^{-1})(n)$ to ensure that $\pi_j(y) = z$.
  Next, suppose $\beta$ is a
  $\Gamma_j$-word with $t(\beta) \leq k$ such that we can write $\beta =
  \gamma_{j,l} \alpha$ for some $l < k$ and $\alpha$ which is not
  a $\Gamma_j$-word such that
  $t(\alpha) < t(\beta)$. Player I must define $y(\beta)(i)(n)$ on turn $k$
  where $\<i,n\> = k - t(\beta)$. Now for all $\gamma \in \Gamma_j$,
  we have
  $(\gamma \cdot y)(\gamma_{j,l} \alpha) = y(\gamma^{-1} \gamma_{j,l} \alpha)
  = \gamma_{j,l}^{-1} \gamma \cdot
  y(\alpha)$ where of course $\gamma_{j,l}^{-1} \gamma \in \Gamma_j$.
  Hence, since $t(\alpha) + \<i,n\> < t(\beta) + \<i,n\> = k$, we have
  that $(\gamma \cdot y)(\gamma_{j,l} \alpha)(i)(j)$ has already been defined
  by the induction hypothesis. Thus, we must make the $k$th move for player
  I in the game associated to $\gamma \cdot y$ using this information. Now
  player II responds by making their $k$th move in the games, and so we
  have played the first $k$ turns of the games, and defined $y(\alpha)(i)(n)$
  for all $\alpha \in \fpGammaiI$, and $i, n \in \omega$ such that
  $t(\alpha) + \<i,n\> \leq k$.
\end{proof}

Now we combine the above two lemmas to give the following lemma, which is
part of Theorem~\ref{g1}.

\begin{lemma} \label{equivariant_hom}
  Suppose $I \leq \omega$ and $\{\Gamma_i\}_{i \in I}$ are countable groups.
  Let $\{A_i\}_{i \in I}$ be a Borel partition of $\left(\prod_{i} \cantor
  \right)^{\fpGammaiI}$. Then there exists some $j \in I$ and an injective
  continuous function $f \from \left(\cantor\right)^{\Gamma_j} \to
  \left(\prod_i \cantor \right)^{\fpGammaiI}$ that is $\Gamma_j$-equivariant
  with respect to the shift actions and such that $\ran(f) \subset A_j$. 
\end{lemma}

\begin{proof}
  By Borel determinacy, either player I or player II has a winning strategy
  in each game $G^{A_j}_j$. By Lemma~\ref{player_I_lemma}, player II must win
  $G^{A_j}_j$ for some $j$. Fix this $j$, and a winning strategy in this
  game.

  We now define the equivariant continuous function $f \from \cantor \to
  \left(\prod_i \cantor \right)^{\fpGammaiI}$. We do this using
  Lemma~\ref{player_II_lemma}: let
  $f(x)$ be the unique $y$ such that for all $\gamma \in \Gamma_j$, and all
  $x \in \cantor$, we have $\pi_j(y) = x$, and that 
  $\gamma \cdot y$ is an outcome of the winning strategy for
  player II in the game $G^{A_j}_j$. Now $f$ is injective since
  $\pi_j(f(x)) = x$. 
  The equivariance of $f$ follows
  from the uniqueness property of Lemma~\ref{player_II_lemma} and the
  equivariance of $\pi_j$, which implies that
  $\gamma \cdot f(x)$ and $f(\gamma \cdot x)$ are equal. 
  Finally, since $f(x)$ is a winning outcome of player II's
  strategy in $G^{A_j}_j$, we have that $f(x) \in A_j$ for all $x$. It is
  easy to check from the proof of Lemma~\ref{player_II_lemma} that $f$ is
  continuous. Roughly, the value of each bit of $f(x)$ depends only on
  finitely many moves in finitely many games which depend on only finitely many bits
  of $x$.
\end{proof}

\begin{remark}\label{identity_projection}
  The proof of Lemma~\ref{equivariant_hom} shows that $f$ can be chosen such
  that $\pi_j(f(x)) = x$. 
\end{remark}

\subsection{The free part of the shift action}
\label{subsec:comb_F_omega}

Our next goal is to prove a version of Lemma~\ref{equivariant_hom} for the
free part of the shift action. 
To begin, we recall a 
lemma from~\cite{MR3454384}.
Suppose that $I \leq \omega$ and $\{E_i\}_{i \in I}$ is a collection of at
least two equivalence relations on $X$. Then the $E_i$ are said to be
\define{independent} if there does not exist a sequence $x_0, x_1, \ldots, x_n$
of distinct elements of $X$, and $i_0, i_1, \ldots i_n \in \omega$ with $n
\geq 2$ such that $i_j \neq i_{j+1}$ for $j < n$ and $x_0 \E_{i_0} x_1
\E_{i_1} x_2 \ldots x_n \E_{i_n} x_0$. The \define{join of the $E_i$}, denoted $\bigvee_{i \in I}
E_i$, is the smallest equivalence relation containing all the $E_i$.
Precisely, $x$ and $y$ are $\bigvee_{i \in I} E_i$-related if there is a
sequence $x_0, x_1, \ldots x_n$ of elements in $X$ such that $x = x_0$, $y
= x_n$, and for all $j < n$, we have $x_j \E_i x_{j+1}$ for some $i \in I$.
Finally, we say that the $E_i$ are \define{everywhere non-independent} if
for every $\bigvee_{i \in I} E_i$ equivalence class $A \subset X$, the
restrictions of the $E_i$ to $A$ are not independent.

\begin{lemma}[\cite{MR3454384}*{Lemma 2.3}]\label{partition_non_independent}
Suppose that $I \leq \omega$ and $\{E_i\}_{i \in I}$ are
countable Borel equivalence relations on a standard Borel space $X$ that
are everywhere non-independent. Then there exists a Borel partition
$\{B_i\}_{i \in I}$ of $X$ such that for all $i \in I$, 
$\comp{(B_i)}$ meets every $E_i$-class.
\end{lemma}

We will combine this lemma with one other lemma that we will use to deal with
the non-free part of the action.

\begin{lemma}\label{Y_complement_lemma}
  Suppose that for every $i \in I$, $X_i$ is a $\Gamma_i$-invariant Borel subset of $(\prod_i
  \cantor)^{\fpGammaiI}$. Let $Y$ be the largest invariant set of $y$ such that $y \in X_i$ for all
  $i \in I$. Then there is a Borel partition $\{C_i\}_{i \in I}$ of the
  complement of $Y$ such that if $A \subset C_i$ is $\Gamma_i$-invariant, 
  then $A \inters X_i = \emptyset$. 
\end{lemma}
\begin{proof}
  We will define Borel sets $C_{i,n}$ for $i \in I$ and $n \in
  \omega$ which partition $\comp{Y}$. We will then let $C_i = \bigunion_{n \in
  \omega} C_{i,n}$.
  If $\delta^{-1} \cdot y \notin X_i$, say that the pair $(\delta,i)$
  witnesses $y \notin Y$. Note that if $(\delta,i)$ witnesses $y \notin Y$,
  then $(\gamma^{-1}\delta, i)$ witnesses $\gamma \cdot y \notin Y$.

  Let $C_{i,0}$ be the set of $y \in
  \comp{Y}$ such that
  $(1,i)$ witnesses $y \notin Y$ and there is no $j < i$ such that
  $(1,j)$ witnesses $y \notin Y$. Note that $C_{i,0}$ does not meet $X_i$.
  For $m > 0$, let $C_{i,m}$ be the set of $y \in \comp{Y}$ so that $m$ is
  the 
  minimal length of a $\delta$ so that some $(\delta,j)$ witnesses $y
  \notin Y$, and $i$ is least such that such a $\delta$ may be
  chosen to be a $\Gamma_i$-word. 
  Note that by the length of a word $\delta \in \fpGammaiI$, we
  mean that if $\delta = \gamma_{i_0,k_0} \ldots \gamma_{i_n,k_n}$ with
  $i_m \neq i_{m+1}$ for all $m$, then the length of $\delta$ is $n + 1$. 

  So suppose now that $A \subset C_i$ is $\Gamma_i$-invariant, and for a
  contradiction suppose that $m$ was least such that $(A \inters C_{i,m})
  \inters X_i$ is nonempty. Let $y$ be an element of this set and note that
  $m > 0$ since $C_{i,0}$ does not meet $X_i$. Since $m > 0$, the
  associated witness that $y \notin Y$ must be of the form $(\delta,j)$,
  where $\delta = \gamma_{i_0,k_0} \ldots \gamma_{i_n,k_n}$ is a
  $\Gamma_i$-word in reduced form, so $i_0 = i$. But then
  $(\gamma_{i_1,k_1} \ldots \gamma_{i_n,k_n},j)$ witnesses
  $\gamma_{i_0,k_0}^{-1} \cdot y \notin Y$, and $\gamma_{i_0,k_0}^{-1}
  \cdot y \in A$ since $A$ is $\Gamma_i$-invariant. But this implies that
  $\gamma_{i_0,k_0}^{-1} \cdot y \in C_{i,m'}$ for some $m' < m$ since $A
  \subset C_i$, and $\gamma_{i_1,k_1} \ldots \gamma_{i_n,k_n}$ has length
  strictly less than $\delta$. This contradicts the minimality of $m$.
\end{proof}

We can now prove a version of 
Lemma~\ref{equivariant_hom} for the free part of the action:

\begin{lemma}\label{free_equivariant_hom}
  Suppose $I \leq \omega$ and $\{\Gamma_i\}_{i \in I}$ are countable groups.
  Let $\{A_i\}_{i \in I}$ be a Borel partition of $\Free((\prod_i
  \cantor)^{\fpGammaiI})$. Then there exists some $j \in I$ and an
  injective continuous function $f \from
  \Free((\cantor)^{\Gamma_j}) \to \Free((\prod_i
  \cantor)^{\fpGammaiI})$ that is $\Gamma_j$-equivariant with respect to
  the shift actions and such that $\ran(f) \subset A_j$.
\end{lemma}  
\begin{proof}
  Our idea is to extend our partition $\{A_i\}_{i \in I}$ to cover the
  whole space 
  $(\prod_i
  \cantor)^{\fpGammaiI}$ in such a way that when we apply 
  Lemma~\ref{equivariant_hom} to this partition, the resulting function $f$
  will have the property that $\ran(f \restriction
  \Free((\cantor)^{\Gamma_j})) \subset A_j$. 

Let $X_i$ be the set of $y \in (\prod_i \cantor)^{\fpGammaiI}$ on which
$\Gamma_i$ acts freely. That is, 
$X_i = \{y \in (\prod_i \cantor)^{\fpGammaiI} : \forall 1 \neq \gamma \in
\Gamma_i (\gamma \cdot y \neq y)\}$. 
Let $Y$ be the largest invariant set
of $y$ such that $y \in X_i$ for every $i$ and let $\{C_i\}_{i \in I}$ be a
Borel partition of the complement of $Y$ as in Lemma~\ref{Y_complement_lemma}. 

Note that $\Free((\prod_i
\cantor)^{\fpGammaiI})$ is a subset of $Y$. Let $E_i$ be the
equivalence relation on $Y$ where $x \mathrel{E_i} y$ if there exists a
$\gamma \in \Gamma_i$ such that $\gamma \cdot x = y$. Note that the
equivalence relations $\{E_i\}$ are everywhere non-independent on the
complement $Y \setminus \Free((\prod_i \cantor)^{\fpGammaiI})$. So by
Lemma~\ref{partition_non_independent}, let $\{B_i\}_{i \in I}$ be a Borel
partition of $Y \setminus \Free((\prod_i \cantor)^{\fpGammaiI})$
so that 
$\comp{B_i}$ meets every $E_i$-class. 

Let $A_i' = A_i \union B_i \union
C_i$, so that $\{A_i'\}_{i \in I}$ is a Borel partition of $(\prod_i
  \cantor)^{\fpGammaiI}$, and apply Lemma~\ref{equivariant_hom} to
  obtain a continuous injective equivariant function $f \from
  (\cantor)^{\Gamma_j} \to A_j'$. 
  
 Now $\ran(f \restriction
  \Free((\cantor)^{\Gamma_j}))$ is invariant under the $\Gamma_j$ action
  since $f$ is $\Gamma_j$-equivariant. Thus, $\ran(f \restriction
  \Free((\cantor)^{\Gamma_j}))$ does not meet $B_j$
  (whose complement meets every $E_j$-class). Since $f$ is injective we
  also have that $\ran(f \restriction
  \Free((\cantor)^{\Gamma_j})) \subset X_j$ and hence $\ran(f \restriction
  \Free((\cantor)^{\Gamma_j}))$ does not meet $C_j$ by
  Lemma~\ref{Y_complement_lemma}. Hence, $\ran(f \restriction
  \Free((\cantor)^{\Gamma_j})) \subset A_j$.
\end{proof}

There is an interesting application of this lemma to a fact about
complete sections of $F(\F_\omega,\cantor)$.

\begin{thm}\label{complete_section_Fomega}
  Suppose that $A$ is a Borel complete section of $F(\F_\omega,\cantor)$. Then
  there is an $x \in \Free(\omega^{\F_{\omega}})$ and some subgroup $\Gamma$ of
  $\F_\omega$ so that $\Gamma$ is isomorphic to $\F_\omega$ and $\gamma \cdot
  x \in A$ for every $\gamma \in \Gamma$.
\end{thm}
\begin{proof}
  Let  $\gamma_0, \gamma_1, \ldots$ be an
  enumeration of all the elements of $\F_\omega$, and define $\{A_i \}_{i
  \in \omega}$ inductively by $A_i = \gamma_i \cdot A \setminus (\union_{j
  < i} A_j)$. Note that the $\{A_i\}_{i \in \omega}$ partition
  $\Free\left(\left(\cantor\right)^{\F_\omega}\right)$. Now if we write 
  $\F_\omega$ as $\F_\omega * \F_\omega * \ldots$, and apply
  Lemma~\ref{free_equivariant_hom},
  then there
  is some $A_i$ and corresponding $i$th copy of $\F_\omega$ such that there is
  an equivariant injection from $\Free(\omega^{\F_\omega})$ into
  $\Free((\cantor)^{\F_2 * \F_2 * \ldots})$ whose range is contained in
  $\gamma_i \cdot A$ for some $i$. Let $\Delta$ be this $i$th copy of
  $\F_\omega$, and note then that if $y \in \ran(f)$ then since $\delta \cdot y \in
  \gamma_i \cdot A$ for all $\delta \in \Delta$, then letting $x =
  \gamma_i^{-1} \cdot y$ and $\Gamma = \gamma_i^{-1} \Delta \gamma_i$, we
  are done. To verify, we check that $\gamma_i^{-1} \delta \gamma_i \cdot x = \gamma_i^{-1}
  \delta \gamma_i \cdot (\gamma_i^{-1} \cdot y) = \gamma_i^{-1} \cdot
  (\delta \cdot y) \in A$ since $\delta \cdot y \in \gamma_i \cdot A$.
\end{proof}

\begin{remark}\label{cs_for_f2}
Theorem~\ref{complete_section_Fomega} is true when we replace
$F(\F_\omega,\cantor)$
by $F(\F_2,\cantor)$. This is because there is a subgroup of $\F_2$ isomorphic to
$\F_\omega$ for which we can equivariantly embed $F(\F_\omega,\cantor)$
into $F(\F_2,\cantor)$ (see~\cite{MR1149121}). Having done this, then 
given any Borel complete
section $A$ of $F(\F_2,\cantor)$, in each equivalence class of 
$F(\F_2,\cantor)$ that meets the range of this embedding of
$F(\F_\omega,\cantor)$, 
we can translate $A$ by the least group element that makes $A$ intersect
the range of this embedding to obtain $A^*$. Now pull $A^*$ back under the
embedding and apply Theorem~\ref{complete_section_Fomega}.
Theorem~\ref{complete_section_Fomega} is also true when we replace
$F(\F_\omega,\cantor)$ with $F(\F_\omega,2)$ by \cite{SewardTuckerDrob}.
Finally, by results of Section~\ref{sec:ultrafilters}, we can find a
$\Gamma$ such that $F(\F_\omega,\cantor)
\restriction \{x : \forall \gamma \in \Gamma(\gamma \cdot x \in A )\}$ 
has the same Borel cardinality as $F(\F_\omega,\cantor)$.
\end{remark}

There is a general theme here that
a complete section of the shift action of $\Gamma$ on
$\Free((\cantor)^\Gamma)$ must have a significant amount of structure. Gao,
Jackson, Khrone, and Seward have some other results which fit into this
theme, which are currently in preparation. See also \cite{MarksNote}.

\subsection{Results in reverse mathematics}

In this section, we show that the lemmas we have proved above giving
continuous equivariant functions into Borel partitions of $(\cantor)^{\fp_i
\Gamma_i}$ truly require the use of determinacy in their proofs, in the
sense that these lemmas imply Borel determinacy. We will also show that the
main lemma from \cite{MR3454384} implies Borel determinacy.
These reversals rely on
the following observation:

\begin{lemma}\label{lem:equivariant_preserves_computable}
  Suppose $X, Y \in \{2,3, \ldots, \omega, \cantor\}$, let $\Gamma = \Delta
  = \Z$, and suppose $f \from \Free(X^\Gamma) \to Y^{\Gamma *
  \Delta}$ is a continuous $\Gamma$-equivariant function whose range is not
  a singleton. Then there is a
  continuous function $g \from \cantor \to \Free(X^\Gamma)$ such that for
  all $x \in \cantor$, we have $f(g(x)) \geq_T x$.
\end{lemma}
We are restricting here to the spaces $\{2,3 \ldots, \omega, \cantor\}$ and the groups $\Gamma = \Delta = \Z$ so that we can sensibly talk
about computability in the space $Y^{\Gamma * \Delta}$. More generally, the
same proof will work for any groups $\Gamma$ and $\Delta$ that contain $\Z$
as a subgroup, provided we choose an appropriate way of identifying $\Gamma *
\Delta$ with $\omega$ so that the cosets of $\Z$ are computable. 
\begin{proof}
  We will assume that $X = Y = \cantor$. The proof is similar in the other
  cases. 
  Note that since $f$ is continuous and $\ran(f)$ is not a singleton, 
  we can find basic clopen
  neighborhoods $U_0, U_1$ of $\Free(X^\Gamma)$ so that $f(U_0)$ and
  $f(U_1)$ are disjoint. 
  Furthermore, by refining the basic clopen sets $U_0$ and $U_1$, we may
  assume that this disjointness is witnessed in the 
  following strong way: there is a $\delta \in \Gamma *
  \Delta$, $i \in \omega$, and $j_0 \neq j_1 \in 2$ so that for
  all $x_0 \in U_0$ and $x_1 \in U_1$, 
  \[f(x_0)(\delta)(i) = j_0 \text{ and } f(x_1)(\delta)(i) = j_1.\]
  Since $U_0$ and $U_1$ are basic clopen neighborhoods in $X^\Gamma$, there are
  finite sets $S_0, S_1 \subset \Z$ and $s_{i,n} \in 2^{< \omega}$ for
  $i \in \{0,1\}$ and $n \in S_i$
  so that
  $U_i = \{x \in X^\Z : \forall n \in S_i : x(n) \supset s_{i,n}\}$. 
  Choose $k$ larger than twice the
  absolute value of any element of $S_0$ or $S_1$ so that $k' \cdot U_{0}$ and
  $U_1$ are compatible for any $k' \in \Z = \Gamma$ with $|k'| \geq k$. 
  
  Our idea is to code $x \in \cantor$ into $g(x) \in
  \Free(X^\Gamma)$ by putting $nk \cdot g(x) \in U_i$ if and only if $x(n)
  = i$. 
  Then from the value of $(nk \cdot f(g(x)))(\delta)(i)$ (in particular
  whether it is $j_0$ or $j_1$), we will be able to determine whether $nk
  \cdot g(x) \in U_0$ or $nk \cdot g(x) \in U_1$ and hence the value of
  $x(n)$. This coding only constrains our choice of $g(x)(m)$ for $m \geq -k$.
  Hence, we can define $g(x)(m)$ for $m < -k$ so that the sequence $g(x)(m)$ for $m < -k$ is not periodic.
  This will ensure that $g(x) \in \Free(X^\Gamma)$. It is clear that we can
  find a continuous map $g \from \cantor \to \Free(X^\Gamma)$ with this property.
\end{proof}

Recall that \define{Turing determinacy} for a pointclass $\Lambda$ 
is the
statement that every Turing invariant set $A \in \Lambda$ either contains a
Turing cone, or its complement contains a Turing cone. Turing determinacy
is closely connected to determinacy in general. For example,
$\mathbf{\Sigma}^1_1$ Turing determinacy is equivalent to
$\mathbf{\Sigma}^1_1$ determinacy by Martin~\cite{MR0258637} and
Harrington~\cite{MR518675}, and Woodin has shown that Turing
determinacy is equivalent to $\AD$ in $L(\R)$. For Borel sets, a close
analysis of Friedman's work in~\cite{MR0284327} shows that over $\mathrm{Z}^- +
\Sigma_1-\text{replacement} + \DC$, Borel Turing determinacy implies Borel
determinacy (see~\cite{MR1008089}, \cite{MartinBook}*{Exercises
2.3.6-2.3.11}, and \cite{Hachtman}).

\begin{thm}\label{thm:reversal}
  The following theorems are each equivalent to Borel determinacy over
  the base theory $\ZZ^- + \Sigma_1-\text{replacement} + \DC$:
  Lemmas~\ref{equivariant_hom} and \ref{free_equivariant_hom} and hence
  Theorems~\ref{g1} and \ref{g2} from this paper, and Lemmas 2.1 and
  3.12 from \cite{MR3454384}. 
\end{thm}
\begin{proof}
  Let $\Gamma = \Delta = \Z$ and suppose $X,Y \in \{2,3, \ldots, \omega,
  \cantor\}$ is appropriate for the theorem or lemma in question
  which we would like to reverse. In the case of the
  Lemmas~\ref{equivariant_hom} and \ref{free_equivariant_hom} in this
  paper, assume that $I = 2$, and $\{\Gamma_i\}_{i
  \in I} = \{\Gamma,\Delta\}$. 

  Suppose $A \subset Y^{\Gamma * \Delta}$ is
  Turing invariant. Since Borel Turing determinacy implies Borel
  determinacy over $\ZZ^- + \Sigma_1-\text{replacement} + \DC$, it will
  suffice to prove that $A$ either contains a Turing cone, or is disjoint
  from a Turing cone. By applying our theorem or lemma, we can find 
  an equivariant continuous map $f$ whose range is contained in $A$ or the
  complement of $A$. The domain of $f$ will at least include 
  $\Free(X^\Gamma)$, and the codomain will be $Y^{\Gamma *
  \Delta}$ for some $X, Y \in \{2,3,\ldots, \omega, \cantor\}$. 
  Hence by Lemma~\ref{lem:equivariant_preserves_computable}, we can find
  a continuous function $g \from \cantor \to \Free(X^\Gamma)$ or $g \from \cantor
  \to \Free(X^\Delta)$ so that for every $x$, $f(g(x)) \geq_T x$. Now
  since $f \circ g$ is continuous, on the cone of $x$ above a code for $f
  \circ g$, $f(g(x)) \equiv_T x$. Hence, the range of $f \circ g$ contains
  representatives of a Turing cone. Hence either $A$ or its complement contains a cone. 
\end{proof}

It is very natural to ask
about the strength in reverse mathematics of the main theorems in both this
paper and 
\cite{MR3454384}, since they are proved using these lemmas which reverse to
Borel determinacy. These are all open problems.\footnote{After a preprint of this paper was posted, Sherwood Hachtman
resolved these problems. 
In particular, Questions~\ref{rq1} and \ref{rq2}
have negative answers. 
Hachtman's proof uses the fact that all these statements have the following
syntactic form: for all Borel
functions $f$, there exists a real $x$ such that $R(x,f)$ where $R$ is a
Borel condition. Thus, for each $f$, each instance is $\Sigma^1_1$ in a
real code for $f$, and
thus these statments are true in levels of $L$ that are $\Sigma^1_1$
correct. This includes levels of $L$ that witness the failure of Borel (or
even $\Sigma^0_4$) determinacy.}
We draw attention to a
pair of questions which we find particularly interesting. 

\begin{question}\label{rq1}
  Does Theorem~\ref{cor:diagonalization} imply Borel determinacy over
  $\ZZ^- + \Sigma_1-\text{replacement} + \DC$? 
\end{question}

Theorem~\ref{cor:diagonalization} stands out to us here because it
resembles (at least in a superficial way) Borel diagonalization theorems
which are known to have strength in reverse mathematics. 

\begin{question}\label{rq2}
  Does \cite[Theorem 3.7]{MR3454384} imply Borel determinacy over
  $\ZZ^- + \Sigma_1-\text{replacement} + \DC$? 
\end{question}

Note that by the proof of Theorem~\ref{thm:reversal} above, \cite[Lemma
2.1]{MR3454384} implies Borel determinacy in the case when $\Gamma$ and
$\Delta$ contain $\Z$. 
However, in the case when $\Gamma$ and $\Delta$ are finite, the
strength of \cite[Lemma
2.1]{MR3454384} is open, and it is in fact equivalent to \cite[Theorem
3.7]{MR3454384}. This is because the spaces $\Free(\N^\Gamma)$ and $\Free(\N^\Delta)$
are countable, so constructing equivariant functions is trivial once we
know $A$ or its complement contain infinitely many $\Gamma$-orbits or
$\Delta$-orbits. Note that we know there cannot be an easy measure
theoretic or Baire category proof of \cite[Theorem
3.7]{MR3454384} by \cite[Theorem 4.5]{MR3454384}, except in the trivial
case $\Gamma = \Delta = \Z/2\Z$. So there is at least some evidence
that \cite[Theorem 3.7]{MR3454384} is hard to prove. 

\pagebreak

\section{Uniform universality}
\label{sec:ri}

\subsection{Introduction to uniform universality}
\label{subsec:intro_to_uu}

In this section, we will investigate a strengthened form of
universality for countable Borel equivalence relations.
The key idea will be to
restrict the class of Borel reductions we consider witnessing $E \leq_B F$
to only those reductions $f$ where there is a way of transforming a witness that $x$ and $y$
are $E$-equivalent into a witness that $f(x)$ and $f(y)$ are $F$-equivalent in a way that is independent of $x$ and $y$. 
To this end, we
will begin this section with a discussion of 
how countable Borel equivalence
relations may be generated. 
Indeed, though the Feldman-Moore theorem~\cite{MR0578656}
implies that every countable Borel equivalence relation can be generated by
the Borel action of a countable group, we will prefer to work with
a more general way of generating equivalence relations since many
equivalence relations from computability theory are not naturally generated by
group actions. Such equivalence relations will play a key role in many
of the theorems we will prove.

Let $E$ be a 
countable Borel equivalence relation on a standard Borel
space $X$.
Then by Lusin-Novikov uniformization \cite{MR1321597}*{Theorem 18.10}, there
exists a countable set $\{\varphi_i\}_{i \in \omega}$ of partial Borel
functions $\varphi_i \from X \to X$ such that $x \mathrel{E} y$ if and only if there is an
$i$ and $j$ such that $\varphi_i(x) = y$ and $\varphi_j(y) = x$.
Conversely, if $X$ is a standard Borel space and $\{\varphi_i\}_{i \in \omega}$ is a countable set
of partial Borel functions on $X$ that is closed under composition and
includes the identity function, then we define $E^X_\ph$ to be the
equivalence relation generated by the functions $\{\varphi_i\}_{i \in \omega}$,
where $x \mathrel{E^X_\ph} y$ if there exists an $i$ and $j$
such that $\varphi_i(x) = y$ and $\varphi_j(y) = x$, in which case we say
\define{$x \mathrel{E^X_\ph} y$ via $(i,j)$}. 
For example, the Turing reductions are a countable set of partial Borel
functions on $\cantor$ which generate Turing equivalence.

Our assumption here that the set $\ph_{i \in \omega}$ is closed under
composition is merely a convenience so that we do not have to discuss
words in the functions $\ph_{i \in \omega}$. For this reason, we will 
assume throughout this section that there is also a computable function $u
\from
\omega^2 \to \omega$ such that $\varphi_i \circ \varphi_j =
\varphi_{u(i,j)}$ for all $(i,j) \in \omega^2$. 

We codify the above into the following convention:

\begin{convention}\label{rec_convention}
Throughout this section, we will let $\ph_{i \in \omega}$ and $\thet_{i \in
\omega}$ denote countable sets
of partial functions on some standard Borel space that contain the identity
function, and are closed under composition as witnessed by some computable
function on indices. 
We will often omit the
indexing on $\ph_{i \in \omega}$ and $\thet_{i \in \omega}$
for clarity.
We will let $E^X_\ph$ and $E^Y_\thet$ denote equivalence
relations on some standard
Borel spaces $X$ and $Y$ that are generated by $\ph$ and $\thet$
respectively.
\end{convention}

Given countable Borel equivalence relations $E^X_\ph$ and $E^Y_\thet$
which are generated by $\ph$ and $\thet$, say that a homomorphism $f \from X \to
Y$ from $E^X_\ph$ to $E^Y_\thet$ is \define{uniform (with respect to $\ph$
and $\thet$)} if there exists a function $u \from \omega^2 \to \omega^2$ such that
for all $x, y \in X$, if $x \mathrel{E^X_{\ph}} y$ via $(i,j)$, then $f(x)
\mathrel{E^Y_{\thet}} f(y)$ via $u(i,j)$, independently of what $x$ and $y$ are.

Now suppose $\Gamma$ is a
countable group equipped with a Borel action on a standard Borel space 
$X$ yielding the countable Borel equivalence relation $E^X_\Gamma$. In this
case we can equivalently regard
$E^X_\Gamma$ as being generated by the 
functions $x \mapsto \gamma \cdot x$ for each $\gamma \in \Gamma$, and so we can
apply our definitions as above in this setting. However, since all of these
functions have inverses, the definitions can be simplified a bit. For
examples, if $E^Y_\Delta$ is generated by a Borel action of the countable
group $\Delta$ on $Y$ and $f$ is a homomorphism from 
$E^X_\Gamma$ to $E^Y_\Delta$, then $f$ is uniform if and
only if there is a function $u \from \Gamma \to \Delta$ such that $u(\gamma)
\cdot f(x) = f(\gamma \cdot x)$ for all $\gamma \in \Gamma$. Hence, in the
setting where the action of $\Delta$ is free, a homomorphism is uniform if
and only if the cocycle associated to it has no dependence on the value of
$x \in X$. This type of cocycle superrigidity is well studied and has many
applications in the field of Borel equivalence relations, as mentioned in
the introduction.

We are ready to give one of the central definitions of this section:
\begin{defn}\label{defn:uniformly_universal}
  A countable Borel equivalence relation $E^X_{\ph}$ generated by $\ph_{i
  \in \omega}$ is said to be \define{uniformly universal (with respect to
  $\ph_{i \in \omega}$)} if for every countable Borel equivalence relation
  $E^Y_\thet$, there is a Borel reduction $f$ from $E^Y_\thet$ to $E^X_\ph$
  that is uniform with respect to $\thet$ and $\ph$.
\end{defn}

The idea of uniform universality was introduced by
Montalb\'an, Reimann, and Slaman (who restricted themselves to the case of
equivalence relations generated by Borel actions of countable groups). They
showed in unpublished work that Turing equivalence is not uniformly
universal with respect to some natural way of generating it by a group.
That Turing equivalence is not uniformly universal as it is usually generated
with Turing reductions is an easy consequence of Slaman and Steel's work
in~\cite{MR960895}.

Note that when we discuss uniform universality, it is important for us to specify the functions $\ph_{i \in \omega}$ that
we use to generate 
the equivalence relation $E^X_\ph$. In particular, we will show that every universal
countable Borel equivalence relation is uniformly universal with respect to
some way of generating it (see Proposition~\ref{some_family}). 

To date, every known universal countable Borel equivalence relation $E$ has
been shown to be universal using a proof that is uniform in the sense of
Definition~\ref{defn:uniformly_universal} (and with respect to some natural way
of generating $E$).
Thus, we can regard the class of uniformly universal countable
Borel equivalence relations as those which we can hope to prove universal
without dramatically new techniques. Indeed, one could make the following
ridiculously optimistic conjecture:

\begin{conj}\label{conj:uniformly_universal}
  If $E$ is a universal countable Borel equivalence relation, then $E$ is
  uniformly universal with respect to every way of generating $E$.
\end{conj}

This conjecture is attractive since we understand uniformly
universality far better than mere universality; 
Theorems~\ref{uu_properties} and 
\ref{uu_classifications} from the introduction would settle many open questions
about universal countable Borel equivalence relations if
Conjecture~\ref{conj:uniformly_universal} were true.

The various parts of Theorems~\ref{uu_properties} and
\ref{uu_classifications} will come from applying some of the tools of
Section~\ref{sec:main_idea} together with an analysis of some particular
examples of natural equivalence relations, many of them from computability
theory. 

We finish by mentioning that there is another view one could take here
which is more computability-theoretic. Instead of studying equivalence
relations, we could instead study locally countable quasiorders (which are
often called 
reducibilities). See \cite{MR3305823}. If $X$ is a standard Borel space, and $\ph_{i \in \omega}$ is a countable
collection of partial functions on $X$ that contains the identity
function and is closed under composition, then we let $\leq_{\ph}$ be the
associated quasiorder where $x \leq_{\ph} y$ if there exists an $i \in
\omega$ such that $\varphi_i(y) = x$. In computability theory, there are several
important examples of uniform embeddings between natural quasiorders from
computability theory. For example, Turing reducibility $\leq_T$ embeds into many-one
reducibility $\leq_m$ via the map $x \mapsto x'$, and Turing reducibility also embeds into enumeration
reducibility via the map $x \mapsto x \join \overline{x}$. These
embeddings are uniform in the sense that if $x \leq_T y$ via the $i$th
Turing reduction $\varphi_i(y) = x$, then there is some many-one
reduction/enumeration reduction $\theta_{u(i)}$ depending only on the index $i$ so that the
images of $x$ and $y$ are related by $\theta_{u(i)}$. One might want, then,
to study the general question of how the usual reducibilities from
computability
theory are related under uniform Borel embedding/reduction.

All of our proofs in Section~\ref{sec:ri} work in this
context of locally countable quasiorders. So for example, many-one reducibility on $3^\omega$ is a uniformly universal
locally countable quasiorder (as are poly-time Turing reducibility and
arithmetic reducibility by the proofs of \cite{MarksPoly} and \cite{1109.1875}). However,
any locally countable Borel quasiorder on $\cantor$ coarser than one-one
reducibility and closed under countable uniform joins is not a uniformly
universal locally countable quasiorder.

\subsection{Basic results on uniform universality}

We will begin by proving some basic facts about 
uniform universality. Since the composition of two uniform homomorphisms is
uniform, it is clear that if $E^X_\ph$ is uniformly universal and there is a uniform reduction
of $E^X_\ph$ to $E^Y_\thet$, then $E^Y_\thet$ is also uniformly universal. 
Hence, to demonstrate
that some equivalence relation is uniformly universal, is enough to show
that we can uniformly reduce a single 
uniformly universal equivalence relation to it. For
this purpose, we explicitly show below that both $E(\F_\omega, \cantor)$
and $E(\F_2,\cantor)$ are 
uniformly universal.

\begin{prop}\label{uniform_examples}
  $E(\F_\omega,\cantor)$ and $E(\F_2,\cantor)$ are both uniformly universal
  (as generated by the left shift actions). 
\end{prop}
\begin{proof}
  This proposition simply follows from the proofs of the universality of
  these equivalence relations given in~\cite{MR1149121}. We will
  recapitulate this argument to explicitly demonstrate how these
  constructions are uniform, and because we will eventually require a
  careful analysis of its details. 
  
  Fix some equivalence relation $E^Y_\ph$ generated by $\ph_{i \in \omega}$
  which we wish to uniformly reduce to $E(\F_\omega,\cantor)$. We may
  assume that $Y \subsetneq \cantor$ is a strict subset of $\cantor$, by
  exploiting the isomorphism theorem for standard Borel spaces. Let $p$ be
  some distinguished point in $\cantor \setminus Y$, and let $Y^* = Y
  \union \{p\}$.

  Let the countably many generators of the group $\F_\omega$ be
  $\{\gamma_{i,j}\}_{(i,j) \in \omega^2}$, exploiting some
  bijection between $\omega$ and $\omega^2$.
  For each $(i,j) \in \omega^2$ define the function 
  $\theta_{\gamma_{i,j}} \from Y^* \to Y^*$: 
  \[\theta_{\gamma_{(i,j)}}(y) = \begin{cases} \theta_i(y) & \text{ if $y \in Y$
  and $\theta_j(\theta_i(y)) = y$} \\ p & \text{ otherwise}\end{cases}\]
  Define $\theta_{\gamma_{(i,j)}^{-1}}$ to be $\theta_{\gamma_{(j,i)}}$.
  Finally, we can define $\theta_w$ for any reduced word $w \in F_{\omega^2}$ by
  composing the $\theta_{\gamma_{(i,j)}}$ and
  $\theta_{\gamma_{(i,j)}^{-1}}$ in the obvious way. Let $\theta_1$ be the
  identity function. 
  
  We define our reduction $f \from
  Y \to \left(\cantor \right)^{\F_\omega}$ from $E^Y_\thet$ to $E(\F_\omega,
  \cantor)$ by $f(y)(\alpha) = \theta_\alpha(y)$. This is a
  uniform reduction; if $x E^Y_\thet y$ via $(i,j)$, then $f(x)
  E(\F_\omega,\cantor) f(y)$ via the generator $\gamma_{(i,j)}$. 

  For the case of $E(\F_2,\cantor)$, let $p$ be a distinguished point in
  $\cantor$, and let $\rho \from \F_\omega \to \F_2$ be an embedding
  of the group $\F_\omega$ into $\F_2$. Now we define our uniform Borel embedding $g \from
  \left(\cantor\right)^{\F_\omega} \to
  \left(\cantor\right)^{\F_2}$ from $E(\F_\omega,\cantor)$ to
  $E(\F_2,\cantor)$ by: 
  \[
    g(x)(\alpha) = 
    \begin{cases} x(\rho^{-1}(\alpha)) & \text{ if $\alpha \in \ran(\rho)$} \\
    p & \text{ otherwise}
  \end{cases}
  \]
  The uniformity of this reduction is witnessed by $\rho$. 
\end{proof}

Note that since $E(\F_\omega,\cantor)$ is generated by a group action, a
countable Borel equivalence relation $E^X_\ph$ is uniformly universal if
and only if $E(\F_\omega,\cantor)$ is uniformly reducible to $E^X_\ph$ if
and only if there is a uniform reduction of every equivalence relation
generated by a Borel action of a countable group to $E^X_\ph$. Hence, our
more general definition agrees with the original definition of Montalb\'an,
Reimann, and Slaman if we restrict to the special case of equivalence
relations equipped with group actions generating them. 

Another useful consequence of Theorem~\ref{uniform_examples} is the
following, which says that our uniform reductions can always be assumed to
have their uniformity witnessed by a computable function. 

\begin{lemma}\label{recursive_uniformity_lemma}
  If $\E^X_\ph$ is a uniformly universal countable Borel equivalence
  relation, then for every countable Borel equivalence relation
  $\E^Y_\thet$, there is a uniform reduction from $\E^Y_\thet$ to $\E^X_\ph$
  with the additional property that its 
  uniformity function $u \from \omega^2 \to \omega^2$ is computable. 
\end{lemma}
\begin{proof}
This follows by analyzing the proof of Proposition~\ref{uniform_examples} above. 
First, the reduction of $E^Y_\thet$ to $E(\F_\omega,\cantor)$ has a 
computable uniformity function. Second, there is a computable
embedding of $\F_\omega$ into $\F_2$ (where the $n$th generator of
$\F_\omega$ is mapped to $\alpha^n
\beta \alpha^{-n}$
where $\alpha$ and $\beta$ are the two generators of $\F_2$). Hence,
composing these two uniform reductions, we get one from 
$E^Y_\thet$ to $E(\F_2,\cantor)$ with a computable uniformity function. 

Finally, take a uniform reduction from $E(\F_2,\cantor)$ to
$E^X_\ph$, since $E^X_\ph$ is assumed to be uniformly universal. This reduction can also
be assumed to have a computable uniformity function since it is enough just
to know how the uniformity of the two generators of $\F_2$ is
witnessed (recall that by Convention~\ref{rec_convention}, we are assuming that
composition of functions in $\ph$ is witnessed by a computable function). Now we are done: compose the reduction of $E^Y_\thet$ to
$E(\F_2,\cantor)$ with the reduction from $E(\F_2,\cantor)$ to $E^X_\ph$. 
\end{proof}

The above lemma will be used in
Theorem~\ref{thm:closed_under_countable_joins} as part of proving that a
large class of equivalence relations are not uniformly universal. 

We note one final consequence of Proposition~\ref{uniform_examples}.

\begin{prop}\label{increasing_union_1}
  If $E^X_{\{\phi_{0,i}\}_{i \in \omega}} \subset E^X_{\{\phi_{1,i}\}_{i \in \omega}} \ldots$ is an
  increasing sequence of countable Borel equivalence relations that are not
  uniformly universal, then their union $E^X_{\thet}$ is not uniformly
  universal, where $\thet$ is the generating family obtained by closing the
  $\{\phi_{j,i}\}$ under composition.
\end{prop}
\begin{proof}
  Since $E(\F_2,\cantor)$ is uniformly universal and $\F_2$ is finitely
  generated, any uniform reduction from $E(\F_2,\cantor)$ to
  $E^X_{\thet}$ must be contained inside $E^X_{\{\phi_{j,i}\}_{i \in
  \omega}}$ for some $j$ since
  the two generators of $\F_2$ correspond to two pairs of functions that
  are from some $\{\phi_{i,j}\}_{i \in \omega}$.
\end{proof}

Hence, uniformly universal equivalence relations are not ``approximable
from below'' by non-uniformly universal equivalence relations. 
See~\cite{MR2563815} and \cite{1109.1875} for some related results on
strong ergodicity that show that 
under the assumption of Martin's conjecture, the (weakly) universal
countable Borel equivalence relations are ``much larger'' than the non
(weakly) universal ones. 

Next, we show that every universal countable Borel equivalence relation is
uniformly universal with respect to some way of generating it.

\begin{prop}\label{some_family}
  If $E$ is a universal countable Borel equivalence relation, then there is
  some countably family of functions generating $E$ for which it is
  uniformly universal.
\end{prop}
\begin{proof}
  This is a trivial corollary of \cite{1109.1875}*{Theorem 3.6}, that if
  $E$ is a universal countable Borel equivalence relation, then $F
  \embeds_B E$, for every countable Borel equivalence relation $F$. 
  
  To see this, let $F = E(\F_\omega,\cantor)$ which is uniformly universal.
  Then we can embed
  $E(\F_\omega,\cantor)$ into $E$ with an injective Borel function $f$. Now
  take the partial functions generating the image of
  $E(\F_\omega,\cantor)$ on $\ran(f)$ and extend these functions to a larger
  countable set that generates $E$. With respect to this set of generators,
  $E$ is uniformly universal.
\end{proof}

Indeed, by the same argument, there is some group action generating
$E$ with respect to which it is uniformly universal.

We will finish this section with a simple application of the results of
Section~\ref{sec:main_idea} to uniform universality. In particular, we will prove part (1) of
Theorem~\ref{uu_properties}.

\begin{thm}\label{F_2_required}
  Suppose $\Gamma$ is a countable group. Then there exists a Borel action
  of $\Gamma$ on a standard Borel space $X$ such that $E^X_\Gamma$ is
  uniformly universal if and only if $\Gamma$ contains $\F_2$ as a
  subgroup.
\end{thm}
\begin{proof}  
  We begin with the forward implication. Since $E^X_\Gamma$ is uniformly
  universal, there exists a uniform Borel reduction of $F(\F_2 * \F_2 *
  \ldots, \cantor)$ to $E^X_\Gamma$.
  Now by Lusin-Novikov uniformization~\cite{MR1321597}*{Theorem 18.10},
  we can partition $\Free((\cantor)^{\F_2 * \F_2 * \ldots})$ into countably
  many Borel sets
  $\{A_i\}_{i \in \omega}$ such that $f$ is injective on each $A_i$. Now by
  Lemma~\ref{free_equivariant_hom}, let $g$ be an equivariant Borel
  injection of $F(\F_2, \cantor)$ into $F(\F_2 * \F_2 * \ldots, \cantor)
  \restrict A_j$ for some $j$ (i.e. equivariant for the $i$th copy of
  $\F_2$). Then $f \circ g$ is a uniform injective
  Borel homomorphism from $F(\F_2, \cantor)$ to $E^X_\Gamma$. If $u \from \F_2
  \to \Gamma$ witnesses this uniformity, then it is clear that the image of
  the two generators of $\F_2$ under $u$ generates a copy of $\F_2$ inside
  $\Gamma$, since $f \circ g$ is injective. 

  The reverse implication follows from the fact that if $\Gamma$ contains
  $\F_2$ as a subgroup, then $E(\Gamma, \cantor)$ is uniformly
  universal by~\cite{MR1149121}. (Following essentially the same
  argument as that in Proposition~\ref{uniform_examples}.)
\end{proof}

Thomas has previously considered the question of which countable groups
admit Borel actions that generate universal countable Borel equivalence
relations~\cite{MR2914864}. 
From our uniform perspective,
Theorem~\ref{F_2_required} gives a complete answer to this question. 
Note that the theorem above combined with
Conjecture~\ref{conj:uniformly_universal} contradicts Thomas'
Conjecture~\cite{MR2914864}*{Conjecture 1.5} that Burnside groups of
sufficiently high exponent can generate universal countable Borel
equivalence relations. 

Another corollary of Theorem~\ref{F_2_required} is a classification of 
which countable groups $\Gamma$ generate uniformly universal countable
Borel equivalence relations with their shift actions on $2^\Gamma$ and
their conjugacy actions on their subgroups.
If
$\F_2 \leq \Gamma$, then both these actions are
uniformly universal by~\cite{MR1149121} and~\cite{MR1815088}.
Theorem~\ref{F_2_required} implies the converses of
these two theorems are true. Hence, we have parts (1) and (2) of
Theorem~\ref{uu_classifications}.

\subsection{Limitations on controlling countable joins}

In this section we will show how results from Section~\ref{sec:main_idea}
can be used to infer a limitation on our ability to control the
computational
power of countable uniform joins, which we will then use to prove that a
number of equivalence relations from computability theory are not uniformly
universal. 

The general problem of controlling the computational power of finite and
countable joins is a frequent topic of investigation in computability theory.
To introduce our lemma let us first recall two contrasting pieces of
folklore. 

First, suppose $\leq_P$ is a Borel
  quasiorder on $\cantor$ with meager sections (e.g. a quasiorder such that
  for every $y$, $\{x : x \leq_P y\}$ is countable). Then there is a continuous
  injection 
  $f \from \cantor \to \cantor$ such for every $x \in \cantor$ and every finite sequence of $y_0, y_1,
  \ldots y_n \in X$ not containing $x$, we have 
  \[f(x) \nleq_P f(y_0) \oplus \ldots \oplus f(y_n)\]
This is easy to show using a simple Baire category argument (see for
instance the Kuratowski-Mycielski theorem \cite{MR1321597}*{Theorem 19.1}).
In computability theory, $\ran(f)$ is called an \define{independent set} for
$\leq_P$. 

Second, suppose we consider countable joins instead of finite joins. Then
the analogue of the above fact becomes false. 
Let us consider \define{one-one-reducibility} here
for concreteness, where $x \leq_1 y$ if there is a computable injection
$\rho \from \omega \to \omega$ such that $x(n) = y(\rho(n))$.
Now if $f \from \cantor \to \cantor$ is any (not necessarily
Borel) function, then there must exist $x \in \cantor$ and a countable
sequence $y_0, y_1, \ldots \in \cantor$ not containing $x$ such that 
\[f(x) \leq_1 f(y_0) \oplus f(y_1) \oplus \ldots \]
To see this, we may clearly assume $f$ is injective (else the statement is
trivial). But then there must be an $n$ such that there are infinitely many
$y$ with $f(y)(n) = 0$ and infinitely many $y$ such that $f(y)(n) = 1$.
Thus, by taking the union of two such countably infinite sets of $y$ and
permuting this set, we can code any real into the sequence $f(y_0)(n),
f(y_1)(n), \ldots$. 

We now prove a substantial strengthening of this second fact of folklore for Borel
functions $f$. 
Suppose
that instead of arbitrary countable joins, we have the dramatically
more modest goal of controlling for each $x$ the countable join of a single
sequence of $y_i$ that depends in a Borel way on $x$. It turns out that this
too is impossible!

\begin{thm}\label{cor:diagonalization}
  Let $\F_\omega$ be the free group on the $\omega$ many generators
  $\gamma_0, \gamma_1, \ldots$ and let $X = \Free((\cantor)^{\F_\omega})$. 
  To each $x \in X$ we associate the single countable sequence
  $\gamma_0 \cdot x, \gamma_1 \cdot x, \ldots$ which does not include $x$.

  Then for all Borel functions $f \from X \to \cantor$, there exists an $x
  \in X$ such that 
  \[f(x) \leq_1 f(\gamma_0 \cdot x) \oplus f(\gamma_1 \cdot x) \oplus \ldots\]
  Indeed, there is an $x$ such that 
  $f(x)(i) = f(\gamma_i \cdot x)(i)$ for all $i \in \omega$.
\end{thm}

\begin{proof}
  Let $f$ be any Borel function from $X$ to $\cantor$,
  and let $A_i$ be the set of $x \in X$ such that $i$ is the least element
  of $\omega$ such that $f(x)(i) \neq f(\gamma_i \cdot x)(i)$. Assume for a
  contradiction that the sets $A_i$ partition $X$. Then by
  Lemma~\ref{free_equivariant_hom} there
  exists some $i$ such that there is a $\<\gamma_i\>$-equivariant Borel
  injection $h$ from
  $\Free((\cantor)^{\<\gamma_i\>})$ to $X$ such that
  $\ran(h) \subset A_i$. Note then that for all $x$, we have $f(h(x))(i)
  \neq f(h(\gamma_i \cdot x))(i)$. But $x
  \mapsto f(h(x))(i)$ would then give a Borel $2$-coloring of the graph
  $G(\Z,\cantor)$. This is easily seen to be impossible with an ergodicity
  argument. (See \cite{MR1667145}).
\end{proof}

Of course $(\cantor)^{\F_\omega}$ is homeomorphic to $\cantor$,
and so we could regard $X$ above as a subset of $\cantor$.

Now it turns out that controlling countable joins in the way
shown to be impossible by Theorem~\ref{cor:diagonalization}
shows up as a subproblem in many natural 
constructions aimed at showing certain equivalence relations from
computability theory are uniformly universal. Our next goal will be to pivot 
Theorem~\ref{cor:diagonalization}
and these failed constructions into a proof that many equivalence relations
from computability theory are not uniformly universal.

Before we finish, we make one more remark: the lemma above
can be restated in graph-theoretic language:

\begin{thm}\label{countable_3_coloring}
  There is a standard Borel space $X$ and countably many $2$-regular Borel
  graphs $\{G_i\}_{i \in \omega}$ on $X$ such that for every Borel
  set $\{c_i\}_{i \in \omega}$ where $c_i$ is a Borel $3$-coloring of $G_i$, 
  there is an $x \in X$ such that
  $c_i(x) = c_j(x)$ for all $i,j$ (i.e. $x$ is \define{$\{c_i\}$-monochromatic}).
\end{thm}
\begin{proof}  
  Let $X = \Free((\cantor)^{\F_\omega})$ and $G_i$ be the graph
  generated by $g_i$ as in Theorem~\ref{cor:diagonalization}. Then given
  countably many $3$-colorings $c_i$, let $f \from X \to \cantor$ be defined by
  $f(x)(i) = c_i(x)$ if $c_i(x) \in \{0,1\}$ and $f(x)(i) = c_i(g_i(x))$ if
  $c_i(x) \notin \{0,1\}$. Now apply Theorem~\ref{cor:diagonalization}.
\end{proof}

We mention this restatement largely for a historical reason:
in~\cite{MarksPhD}, we showed that the above coloring problem was
equivalent to the uniform universality of many-one equivalence. The proof of the
forward direction of this result is essentially contained in the proof of
Theorem~\ref{universality_of_ri}. The proof of the converse of this theorem
makes essential use of a notion of forcing due to Conley and Miller. We
refer the interested reader to \cite{MarksPhD}*{Section 5.2}. The
equivalence of this coloring problem with the uniform universality of
many-one equivalence was 
the catalyst that led the games studied in this paper and thus most of
our results, as well as the results of~\cite{MR3454384}.

\subsection{Uniform universality and equivalence relations from computability theory}
\label{subsec:impossibility}

In this section, we will show that a large class of equivalence relations
from computability theory are not uniformly universal. Beyond the inherent
interest in classifying such equivalence relations our analysis will also
be used to prove the remaining parts of
Theorem~\ref{uu_properties}. We begin with a simple
lemma.

\begin{lemma}\label{diagonal_permutation}
  Suppose $x, y_0, y_1, \ldots \in \cantor$ are each of the form $\join_{i
  \in \omega} z$ for some $z \in \cantor$, and there is a computable
  function $u \from \omega \to \omega$ such that $x \equiv_1 y_i$ via the
  program $u(i)$ for
  all $i \in \omega$. Then there exists a pair of computable bijections
  $r,s \from \omega \to \omega$ such that $x(r(n)) = y_{s(n)}(n)$ for all $n$.
\end{lemma}
\begin{proof}
  The point of this lemma is that $x$ looks like a ``diagonal'' of
  the join of the $y_i$ after permuting $x$ by $r$, and the $y_i$ by $s$. 
  

  Our argument is a simple back and forth construction. At each stage of the
  construction we will have defined $r$ and $s$ on the same finite domain. 

  At even steps, we begin by picking the least $n$ not in the domain of $r$
  and $s$. Let $m$ be the least number not in the range of $s$, and define
  $s(n) = m$. Now since $x \equiv_1 y_m$, we have that $y_m(n) =
  x(\<i,j\>)$ for some $\<i,j\> \in \omega$. It is possible that $\<i,j\>$ is already
  in the range of $r$ but we can always find some $\<i^*,j\>$ not already
  in the range of $r$, and set $r(n) = \<i^*,j\>$, since $y_m(\<i,j\>) =
  y_m(\<i^*,j\>)$ since $y_m$ is of the form $\join_{i \in \omega} z$ for
  some $z \in \cantor$.

  At odd steps, we pick the least $k$ and $m$ that are not in the range of
  $r$ and $s$ respectively. Then since $x$ is recursively isomorphic to
  $y_m$, there is some $\<i,j\>$ such that $x(k) = y_m(\<i,j\>)$. Again,
  we can find some $\<i^*,j\>$ not already in the domain of $r$ and $s$
  and set $r(\<i^*,j\>) = k$ and $s(\<i^*,j\>) = m$. 
\end{proof}

It is a fundamental property of Turing
reducibility that if we computably specify countably
many Turing reductions, then we can run them all simultaneously to produce
the uniform join of their outputs. It is this idea which we encapsulate
into our next definition:

\begin{defn}
  Suppose that $E^{\cantor}_\ph$ is a countable Borel equivalence relation
  on $\cantor$ generated by $\ph_{i \in \omega}$.
  Say that $E$ is \define{closed under countable uniform joins} if whenever
  $x, y_0, y_1, \ldots \in \cantor$, $u \from \omega \to \omega^2$ is
  computable, and
  $x E y_i$ via $u(i)$ for all $i \in \omega$, then $x E \bigjoin_{i \in
  \omega} y_i$.
\end{defn}

For example, many-one equivalence, $tt$ equivalence, Turing
equivalence, and enumeration equivalence all have this property, as
generated by their usual family of reductions.

\begin{thm}\label{thm:closed_under_countable_joins}
  Suppose that $E_\ph$ is a countable Borel equivalence relation
  on $\cantor$ that is coarser than recursive isomorphism and
  is closed under countable uniform joins. Then $E_\ph$ is not
  uniformly universal. 
\end{thm}

\begin{proof}
Let $\F_\omega = \<\gamma_i: i \leq \omega\>$ be the free group on the
generators 
$\gamma_0, \gamma_1 \ldots, \gamma_\omega$.
Let $\Gamma$ be the subgroup
$\<\gamma_i \gamma_\omega^{-1}: i < \omega\>$, which is isomorphic to
$\F_\omega$. Note that
$\gamma_\omega \notin \Gamma$.
Let $X = \Free((\cantor)^{F_\omega})$, and let 
$F_\Gamma \subset F(\F_\omega,\cantor)$ be the equivalence relation on $X$ where $x
\mathrel{F_\Gamma} y$ if
$\alpha \cdot x = y$ for some $\alpha \in \Gamma$. It is interesting to note that $F_\Gamma$ is
Borel isomorphic to $F(\F_\omega,\cantor)$ and is hence a universal
treeable countable Borel equivalence relation.

Now for a contradiction, let $f$ be a uniform Borel reduction of $F_\Gamma$
to $E_\ph$ with a computable
uniformity function by Lemma~\ref{recursive_uniformity_lemma}. So for all
$x \in X$ and $\alpha \in F_\Gamma$ we have
that $f(x) E f(\alpha \cdot x)$ uniformly in $\alpha$. Let $\alpha_0,
\alpha_1, \ldots$ be a computable listing of all the elements of $\Gamma$ in which
each element appears infinitely many times. Let $\hat{f}(x) = \bigjoin_{i
\in \omega} f(\alpha_i \cdot x)$. 
Then we have $f(x) E
\hat{f}(x)$ since $E$ is closed under countable uniform joins.
It is clear that $\hat{f}(\gamma_\omega \cdot x)$ and $\hat{f}(\gamma_i
\cdot x)$ are recursively isomorphic for all $i \in \omega$ by permuting
columns. Further, since each group element appears infinitely often in our
listing $\{\alpha_i\}_{i \in \omega}$ of $\Gamma$, we can apply
Lemma~\ref{diagonal_permutation} to obtain a pair of computable bijections
$r,s \from \omega \to \omega$ such that $\hat{f}(\gamma_\omega \cdot x)(r(i)) =
\hat{f}(\gamma_{s(i)} \cdot x)(i)$ for all $x \in X$ and all $i \in
\omega$.

Now for all $x \in X$, since $\gamma_\omega \notin \Gamma$, we have that $x
\mathrel{\cancel{F_\Gamma}} \gamma_\omega \cdot x$, and so 
$f(x) \mathrel{\cancel{E}} f(\gamma_\omega \cdot x)$, which implies $\hat{f}(x)
\mathrel{\cancel{E}} \hat{f}(\gamma_\omega \cdot x)$ and thus $\hat{f}(x)$ and
$\hat{f}(\gamma_\omega \cdot x)$ are not recursively isomorphic. Hence, for
each $x \in X$, we have that $\hat{f}(x)(i) \neq \hat{f}(\gamma_{s(i)})(i)$
for some $i \in \omega$, since otherwise $\hat{f}(x)$ and
$\hat{f}(\gamma_\omega \cdot x)$ would be recursively isomorphic via $r$.
We will apply the same idea as the proof of
Theorem~\ref{cor:diagonalization} to obtain a contradiction.

For each $i \in \omega$, let $A_{s(i)}$ be the set of $x \in X$ such that
$i$ is the least element of $\omega$ such that $\hat{f}(x)(i) \neq
\hat{f}(\gamma_{s(i)} \cdot x)(i)$. By 
Lemma~\ref{free_equivariant_hom}, there is some
$i$ and a $\<\gamma_{s(i)}\>$-equivariant Borel function $g$ from
$\Free\left(\left(\cantor\right)^{\<\gamma_{s(i)}\>}\right)$ to
$\Free\left(\left(\cantor\right)^{\F_\omega}\right)$ such that $\ran(g) \subset
A_{s(i)}$. Then $x \mapsto \hat{f}(g(x))(i)$ yields a Borel $2$-coloring of
$G(\Z,\cantor)$, which is a contradiction.
\end{proof}

Later in this section, we will show that this theorem is specific to
$\cantor$; it is not true when $\cantor$ is changed to $3^\omega$. In
particular, we will show that the equivalence relation of many-one
equivalence on $3^\omega$ is a uniformly universal countable Borel
equivalence relation. 

We also remark that the statement of the theorem can
be strengthened slightly:

\begin{remark}\label{remark:careful}
The proof of Theorem~\ref{thm:closed_under_countable_joins} yields a statement that is
actually slightly stronger than what we have stated. Because there is a
uniform embedding of $F(\F_\omega,\cantor)$ into $F(\F_2,\cantor)$,  
the countable uniform joins that were used in this proof are very simple:
they are computable compositions of two pairs of elements of $\ph$. 
Hence Theorem~\ref{thm:closed_under_countable_joins} remains true when we
assume closure under this smaller class of countable uniform joins.
\end{remark}

It is interesting that our proof of
Theorem~\ref{thm:closed_under_countable_joins} only uses that fact that the
universal treeable equivalence relation can not be uniformly reduced to
$E_\ph$. Very little is known about what countable Borel equivalence
relations can be reduced to Turing equivalence, or any other such $E_\ph$
satisfying the hypothesis of
Theorem~\ref{thm:closed_under_countable_joins}. We make the following
conjecture:

\begin{conj}\label{no_universal_treeable}
  The universal treeable countable Borel equivalence relation is not Borel reducible to
  Turing equivalence.
\end{conj}

Next, we will turn specifically to the case of equivalence relations
coarser than Turing equivalence. We will begin by considering the following
examples:

\begin{defn}\label{natural_coarse}
Let $\alpha$ be an ordinal less than $\omega_1^{ck}$ that is additively
indecomposable so that if $\beta_0, \beta_1 < \alpha$ then $\beta_0 +
\beta_1 < \alpha$. Define \define{$(< \alpha)$-reducibility}, noted 
$\leq_{(<\alpha)}$, by $x \leq_{(<\alpha)} y$ if and only if
there exists a $\beta < \alpha$ such that $x \leq_T y^{(\beta)}$ where
$y^{(\beta)}$ is the $\beta$th iterate of the Turing jump relative to $y$. 
(Our assumption here that $\alpha$ is additively indecomposable
is needed so that $\leq_{(< \alpha)}$ is transitive).
The symmetrization of this reducibility is the equivalence relation
$\equiv_{(<\alpha)}$.
\end{defn}

Hence, $\leq_{(<1)}$ is Turing
reducibility, arithmetic reducibility is $\leq_{(<\omega)}$, and
so on. Now $(< \alpha)$-reducibility is naturally generated by the
functions obtained by taking the Turing reductions and the
functions $x \mapsto x^{(\beta)}$ for all $\beta < \alpha$, and closing under
composition. We assume henceforth that $(< \alpha)$-equivalence is
generated by these canonical functions. 

Note that we can relativize this definition to any $x \in \cantor$, and
every additively indecomposable $\alpha < \omega_1^x$. Of course, different
$x$ yield different equivalence relations, but any two such definitions of $(<
\alpha)$-equivalence relations for the same ordinal $\alpha$ 
will agree on a Turing cone.

There is a certain sense in which $(< \alpha)$-reducibilities are the only 
``natural'' computability-theoretic reducibilities coarser
than $\leq_T$. We may justify this the following way. 
Suppose that $\leq_P$ is any countable Borel quasiorder
that is coarser than Turing equivalence and closed under finite computable
joins, i.e. if $x \geq_P y, z$, then $x \geq_P y \oplus z$.
Then Slaman~\cite{MR2143887} has shown that on a Turing cone, $\leq_P$ is
$(< \alpha)$-reducibility relative to $x$ for some $x$ and $\alpha <
\omega_1^x$. Thus, up to a Turing cone, 
these reducibilities are the only ones 
coarser than Turing reducibility that are closed under
finite computable joins.

We give a complete description of which of these equivalence
relations are uniformly universal.

\begin{thm}\label{coarser_classification}
  Suppose $\alpha < \omega_1^{ck}$ is additively indecomposable. Then
  $(< \alpha)$-equivalence is uniformly universal if and only if there
  is a $\beta < \alpha$ such that $\alpha = \beta \cdot \omega$.
\end{thm}

We prove the non-uniform universality here by using
Remark~\ref{remark:careful}. In the other case when $\alpha = \beta \cdot
\omega$ for some $\beta < \alpha$, the proof of universality is an easy
extension of Slaman and Steel's proof that arithmetic equivalence is
universal. Recall from \cite{1109.1875} that given 
$x \from \omega \to 2^{< \omega}$ and $y \in
2^\omega$, we define $J(x,y) \in \cantor$ to be the real whose $n$th
column is 
\[\left(J(x,y) \right)^{[n]} = \begin{cases} 
x(n) 1 0000\ldots & \text{ if y(n) = 0}\\
x(n) 0 1111\ldots & \text{ if y(n) = 1}
\end{cases}
\]

We now give an extension of this definition that iterates this type of jump
coding through the transfinite. 

\begin{defn}\label{jump_code_defn}
  Given $x \from \omega \to 2^{< \omega}$ and $y \in
  2^\omega$, and a notation for $\alpha < \omega_1^{ck}$, we define $J_\alpha(x,y) \in
  \cantor$ as follows. For $\alpha = 1$,
  $J_1(x,y) = J(x,y)$. For $\alpha = \beta + 1$ for $\beta > 0$, we define
  $J_\alpha(x,y) = J(x_0,J_\beta(x_1,y))$, where $x = x_0 \oplus x_1$.
  Finally, 
  suppose now that $\alpha$ is a limit ordinal and $(\lambda_n)_{n <
  \omega}$ is an
  computable sequence of computable ordinals whose limit is the computable
  ordinal $\alpha$.
  Then for $n \in \omega$, define the functions $f_n(x,y,(\lambda_n))$ by
  \[f_n(x,y,(\lambda_n)) = J_{\lambda_n}(x^{[n]},y(n) \concat
  f_{n+1}(x,y,(\lambda_n)))\] 
  where $x^{[n]}$ is the
  $n$th column of $x$. Note that this definition is really an inductive definition of
  the $k$th bit $f_n(x,y,(\lambda_n))(k)$ simultaneously for all $n$; the definition of $f_n(x,y,(\lambda_n))(k)$
  will
  only use the values of $f_{n+1}(x,y,(\lambda_n))(k')$ for $k' < k$. 
  Finally, define $J_\alpha(x,y) = f_0(x,y,(\lambda_n))$. 
\end{defn}

Now we have the following: 

\begin{lemma}\label{alpha_forcing}
  If $x_0, \ldots, x_i$, $z_0, \ldots, z_j$, and $w$ are mutually
  $\Delta^1_1$ generic functions from $\omega$ to $2^{< \omega}$, then
  for all $\alpha, \beta, \gamma < \omega_1^{ck}$ and $y_0, \ldots, y_i \in \cantor$
  \begin{enumerate}
  \item $\displaystyle{\left(0^{(\alpha)} \join J_\beta(x_0, y_0) \join \ldots \join
  J_\beta(x_i, y_i) \join z_0 \join \ldots \join
  z_j\right)^{(\beta)}}$
  \vspace{-.2em}
  \begin{flushright}
  $\displaystyle{\equiv_T 0^{(\alpha+\beta)} \join x_0 \join \ldots \join
  x_i \join y_0 \join \ldots \join y_i \join z_0 \join \ldots
  z_j}$\end{flushright}
  \item $\displaystyle{0^{(\alpha)} \join J_\beta(x_0, y_0) \join \ldots \join
  J_\beta(x_i, y_i) \join z_0 \join \ldots \join
  z_j \ngeq_T w}$
  \end{enumerate}
\end{lemma}
\begin{proof}
  This is an easy extension of the proof of \cite{1109.1875}*{Lemma 2.4}.
  The successor step is essentially identical, and at limits we use the
  fact that the equivalences at each step are proved uniformly in $\beta$. 
\end{proof}

\begin{proof}[Proof of Theorem~\ref{coarser_classification}]
  Suppose $\alpha < \omega_1^{ck}$ is closed under addition and there
  exists a $\beta < \alpha$ such that $\alpha = \beta \cdot \omega$. We
  must show that $\equiv_{(< \alpha)}$ is a universal countable Borel equivalence
  relation. It is enough to show that if $\F_2 = \<a,b\>$ acts on a Polish
  space $X$, then the resulting equivalence relation $E^X_{\F_2} \leq_B
  \equiv_{(< \alpha)}$. 
  Let $g \from X \to \left(2^{< \omega} \right)^{\omega}$ be a
  Borel function such that for every distinct $x_0, \ldots, x_n \in X$, we
  have that $g(x_0), \ldots, g(x_n)$ are all mutually $\Delta^1_1$ generic
  functions from $\omega$ to $2^{< \omega}$. The definition of the Borel
  reduction witnessing $E^X_{\F_2} \leq_B \equiv_\alpha$ is:
  \[f(x) = J_\beta(g(x),f(a \cdot x) \join f(a^{-1} \cdot x) \join f(b
  \cdot x) \join f(b^{-1} \cdot x))\]
  Like Definition~\ref{jump_code_defn}, this is really an 
  inductive definition of
  $f(x)(k)$, which depends on values of $f(\gamma \cdot x)(k')$ for $\gamma
  \in \{a,a^{-1},b,b^{-1}\}$, but only for $k' < k$. 

  By inductively applying Lemma~\ref{alpha_forcing}, we see that 

  \[f^{(\beta \cdot n)}(x) = 0^{(\beta \cdot n)} \join \bigjoin_{\{\gamma
  \in \F_2: |\gamma| < n\}} g(\gamma \cdot x) \join \bigjoin_{\{\gamma \in
  \F_2: |\gamma| = n\}} f(\gamma \cdot x)\]
  where $|\gamma|$ is the length of $\gamma$ as a word in $\F_2$.
  Hence, by part 2 of Lemma~\ref{alpha_forcing}, we see that
  $f^{(\beta \cdot n)}(x)$ can not compute $g(y)$ for any $y$ not in the
  same $E^X_{\F_2}$-class of $x$. Thus, $f$ is an embedding, since
  therefore $f^{(\beta \cdot (n+1))}(x)$ can not compute $f(y)$ for all $n$
  and all $y$ not in the same $E^X_{\F_2}$-class as $x$.

  Now conversely, suppose that for every $\beta < \alpha$, we have that
  $\beta \cdot \omega < \alpha$. Now if we have pair of $(< \alpha)$-reductions, then they must both be $(< \beta)$-reductions for some
  $\beta < \alpha$. But then given any $x$, if we have a countable uniform
  of reals obtained by computably composing these two reductions applied
  to $x$, then it
  is still a $(< \alpha)$-reduction since $\beta \cdot \omega < \alpha$. Hence, by 
  Remark~\ref{remark:careful},
  $(< \alpha)$-equivalence is not uniformly universal.
\end{proof}

This proves part (4) of Theorem~\ref{uu_classifications}.
This also proves part (3) of
Theorem~\ref{uu_properties} since every countable Borel
equivalence relation is contained in $(< \alpha)$-equivalence relative
to some real. 
Note that 
the equivalence relations $\equiv_{(< \omega)}$,
$\equiv_{(< \omega^2)}$,
$\equiv_{(< \omega^3)}$, \ldots are each uniformly universal by
Theorem~\ref{coarser_classification}, but their union
$\equiv_{(< \omega^\omega)}$ is not. This finishes the proof of part (4)
of Theorem~\ref{uu_properties} together with
Proposition~\ref{increasing_union_1}. 

\subsection{Contrasting results in the measure context and on $3^\omega$}

In this section, we shall prove a number of contrasting results to those in
the previous section by showing that some of the equivalence relations
considered there are universal if we change the space from $\cantor$ to
$3^\omega$, or are willing to discard nullsets.
We will begin with some combinatorial results in
the measure context which we will use in one of our constructions.

\begin{lemma}\label{dcs_small}
  Suppose $E$ and $F$ are countable Borel equivalence relations on a
  standard probability space $(X,\mu)$. Then there is a Borel set $A
  \subset X$ that meets $\mu$-a.e. $E$-class of cardinality $\geq 3$, and
  whose complement meets $\mu$-a.e. $F$-class of cardinality $\geq 3$.
\end{lemma}
\begin{proof}
  We may find some standard Borel space $Y \supset X$ and equivalence
  relations $E^*$ and $F^*$ extending $E$ and $F$ such that all the 
  $E^*$-classes and $F^*$-classes have cardinality $\geq 3$ and for all $x
  \in X$ whenever 
  $[x]_E$ has cardinality $\geq 3$, then $[x]_E = [x]_{E^*}$ and when $[x]_F$
  has cardinality $\geq 3$, then $[x]_F = [x]_{F^*}$. Now by
  \cite{MR3454384}*{Theorem 1.7} there is a Borel set $A \subset Y$ such
  that $A$ meets $\mu$-a.e. $E^*$-class and the complement of $A$ meets
  $\mu$-a.e. $F^*$-class.
\end{proof}

From this, we can conclude the following, which shows that
Theorem~\ref{cor:diagonalization} does not hold after discarding a nullset.

\begin{lemma}\label{measure_diagonalization}
  Suppose that $X$ is a standard Borel space, $\{g_i\}_{i \in \omega}$ is a
  countable collection of partial Borel injections $X \to X$, and $\mu$ is
  a Borel probability measure on $X$. Then there is a Borel set $A$ of full
  measure and countably many Borel functions $c_i \from A \to \cantor$ such that
  for all $x \in \cantor$, if $g_i(x)$ is defined and not equal to $x$ for
  all $i$, then there exists an $i$ such that $c_i(x) \neq c_i(g_i(x))$. 
\end{lemma}
\begin{proof}
  We may assume all the $g_i$ are total Borel automorphisms of $X$. This is
  because we may extend the $g_i$ to total Borel automorphisms 
  on some larger standard Borel space $Y \supset X$. 
  
  For each $i \in \omega$, let $E_i$ be the equivalence relation generated
  by $g_i$ and let $G_i$ be the
  graph generated by $g_i$. 

  We may prove this fact with merely two functions 
  $c_0$ and $c_1$. By Lemma \ref{dcs_small}, there is a Borel set $B$ that
  meets $\mu$-a.e. $E_0$
  class of cardinality $\geq 3$ and whose complement meets $\mu$-a.e. $E_1$
  class of cardinality $\geq 3$. Let $G_0^*$ be the graph where we remove
  the edge between $x$ and $g_0(x)$ for every $x \in B$, and let $G_1^*$ be
  the graph where we remove the edge between $x$ and $g_1(x)$ for every $x
  \in \comp{B}$. Now there is a Borel set $A$ of full measure such that
  every connected component of $G_0^* \restrict A$ and $G_1^* \restrict A$
  are finite. Finish by letting $c_0$ be a Borel $2$-coloring of $G_0^*
  \restrict A$ and $c_1$ be a Borel $2$-coloring of $G_1^* \restrict A$.
\end{proof}

We are now ready to prove the following theorem:

\begin{thm}\label{group_ri}
  Suppose $Z$ is some countable set (which is possibly empty) and 
  $G$ is a countable group of permutations of the set $\F_2 \times
  \omega \disjointunion Z$ so that for every $\delta \in \F_2$, there
  exists some $\rho_\delta
  \in G$ so that 
  $\rho_\delta((\gamma,n)) = (\delta \gamma,n)$ for every $(\gamma,n) \in \F_2
  \times \omega$. Then
  \begin{enumerate}
    \item The permutation action of $G$ on $2^{\F_2 \times \omega}$
    generates a measure universal countable Borel equivalence relation.
    \item The permutation action of $G$ on $3^{\F_2 \times \omega}$
    generates a universal countable Borel equivalence relation. 
  \end{enumerate}
\end{thm}
  Suppose briefly that $Z$ is empty and we identify $2^{\F_2 \times
  \omega}$ with $(2^\omega)^{\F_2}$. Then note that the permutations
  $(\gamma,n) \mapsto (\delta \gamma, n)$ which are required to be in the
  group $G$ generate the shift action of $\F_2$ on $(Y^\omega)^{\F_2}$;
  this is their significance.
\begin{proof}
  Throughout we will let $Y \in \{2,3\}$. 
  Let $E_\infty$ be a universal countable Borel equivalence relation
  generated by an action of $\F_2$ on a standard Borel space $X$. If 
  $f \from X \to Y^\omega$ is a function, then define the function 
  $\hat{f} \from X \to Y^{\F_2 \times \omega \disjointunion Z}$ by 
  \[\hat{f}(x)((\gamma,n)) = f(\gamma^{-1} \cdot x)(n)\]
  for $(\gamma,n) \in F_2 \times \omega$, and $\hat{f}(x)(z) = 0$ for all
  $z \in Z$.

  Note that if $x, y \in X$ and $\delta \cdot x = y$, then $\rho_\delta
  \cdot \hat{f}(x) = \hat{f}(y)$, since 
  \begin{multline*}
  \rho_\delta \cdot
  \hat{f}(x)((\gamma,n)) = \hat{f}(x)(\rho_{\delta}^{-1}((\gamma,n)))  =
  \hat{f}(x)((\delta^{-1} \gamma,n))\\ = f(\gamma^{-1} \cdot \delta \cdot
  x)(n) = \hat{f}(\delta \cdot x)((\gamma,n)) = \hat{f}(y)((\gamma,n))
  \end{multline*}
  So given any Borel $f$, the associated $\hat{f}$ is a Borel homomorphism
  from $E_\infty$ to the orbit equivalence relation of the permutation
  action of $G$ on $Y^{\F_2 \times \omega \disjointunion Z}$. 
  We will define a Borel function $f$ so that
  the corresponding $\hat{f}$ becomes our desired Borel reduction. 

  Say a permutation $\rho \in G$ \define{uses $\{n,m\}$} for $n \neq m
  \in \omega$ if there exist group elements $\delta, \gamma \in \F_2$ such
  that $\rho((\gamma,n)) = (\delta,m)$.  
  Let $G'$ be the set of $\rho \in G'$ so $\rho$ does not use infinitely
  many pairs $\{n,m\}$, and there are only finitely many $n$ such that
  $\rho(Z) \union \rho^{-1}(Z)$ meets $\F_2 \times \{n\}$. Note that $G'$
  is a subgroup of $G$.
  
  Let $S_0, S_1 \subset \omega$ be
  disjoint sets so that $S_0 \union S_1$ is coinfinite, and 
  \begin{enumerate}
  \item For every $\rho
  \in G$ that uses infinitely many pairs $\{n,m\}$, there is some pair
  $\{n,m\}$ used by $\rho$ so that $n \in S_0$ and $m \in S_1$. 
  \item For
  every $\rho \in G$ such that 
  there are infinitely many
  $n$ such that $\rho(Z) \union \rho^{-1}(Z)$ meets $F_2 \times \{n\}$, 
  $\rho(Z) \union \rho^{-1}(Z)$ meets $\F_2 \times S_1$. 
  \end{enumerate}

  Our first constraint on the function $f$ will be that for every $x \in
  X$, 
  \[f(x)(n) = 0 \land f(x)(m) = 1 \text{ for every
  $n \in S_0$ and $m \in S_1$}.\]

  We claim now that if $\rho \in G$, but $\rho \notin G'$, then $y \in
  \ran(\hat{f})$ will imply that $\rho
  \cdot y \notin \ran(\hat{f})$. This is because if $\rho((\gamma,m)) =
  (\delta,n)$ and $n \in S_0$ and $m \in S_1$, then
  for all 
  $y \in \ran(\hat{f})$, we have $y((\gamma,m)) = 1$ and $y((\delta,n)) =
  0$. However, $(\rho \cdot y)((\delta,n)) = y((\gamma,m)) = 1$, and so
  $(\rho \cdot y) \notin \ran(\hat{f})$. We argue similarly for $\rho$ such
  that $\rho(Z) \union \rho(Z^{-1})$ meets $\F_2 \times S_1$.

  Let $S_2 \subset \omega$ be an infinite set disjoint from $S_0 \union S_1$
  and so that $S_0 \union S_1 \union S_2$ is coinfinite. Let $h \from X \to
  2^{S_2}$ be a Borel reduction from the equality relation on $X$ to the
  relation of equality mod finite on $2^{S_2}$.
  (For example, we may assume $X = \cantor$, let $\pi \from S_2 \to \omega$
  be an infinite-to-one surjection, and let $h(x)(n) = x(\pi(n))$.) 
  Now define
  \[f(x)(n) = h(x)(n) \text{ for every $n \in S_2$}\tag{*}\]

  Suppose $\rho \in G'$. Then for
  sufficiently large $n \in S_2$, we must have that for every $x \in X$,
  $(\rho \cdot \hat{f})(x)(n)$ has been defined by (*). Because we can recover
  $y \in X$ if we know all but finitely many bits of $h(y)$, there can be
  at most one $y$ such that for all sufficiently large $n \in S_2$, we have 
  $(\rho \cdot \hat{f})(x)(\gamma,n) = \hat{f}(y)(\gamma,n)$ for 
  all $\gamma \in \F_2$. Let
  $g_\rho \from X \to X$ be the partial Borel function mapping each $x \in
  X$ to this unique $y \in X$ if it exists. Note that if $g_\rho(x) = y$
  then $g_{\rho^{-1}}(y) = x$, so $g_\rho$ is a partial injection. 
  For each $\rho$, also define $n_\rho$ to be
  least such that for every 
  $n \geq n_\rho$ and $\gamma \in F_2$ there is a $\delta \in \F_2$ such
  that $\rho((\gamma,n)) = (\delta,n)$. 

  For each $n \geq n_\rho$, let $g_{\rho,n} \from X \to X$ be the partial Borel
  function so $g_{\rho,n}(y) = \gamma \cdot g_{\rho}^{-1}(y)$ where $\gamma
  \in \F_2$ is such that $\rho^{-1}((1,n)) = (\gamma,n)$.
  Hence, if $g_\rho(x) = y$, and $x \mathrel{\cancel{E_\infty}} y$ then $\rho \cdot \hat{f}(x) = \hat{f}(y)$
  would imply 
  $f(y)(n) = \hat{f}(y)((1,n)) = (\rho \cdot \hat{f}(x))((1,n)) =
  \hat{f}(x)(\rho^{-1}(1,n)) = \hat{f}(x)((\gamma,n)) = f(\gamma^{-1} \cdot
  x)(n) = f(g_{\rho,n}(y))(n)$ for every $n \geq n_\rho$.

  Hence, to make $\hat{f}$ a Borel reduction, it suffices to ensure that
  for every $y \in X$
  and every $\rho \in G'$ if
  $g_\rho^{-1}(y)$ is defined, then either 
  $g^{-1}(y) \mathrel{E_\infty} y$, or else there is an $n \geq n_\rho$ such
  that $f(y)(n) \neq f(g_{\rho,n}(y))$. 
  
  For each $\rho \in G'$ let $S_{3,\rho}$
  be an infinite set disjoint from $S_0 \union S_1 \union S_2$ with
  $\min(S_{3,\rho}) \geq n_\rho$, and so that the sets
  $\{S_{3,\rho}\}_{\rho \in G'}$ are all
  pairwise disjoint. Define $f(x)(n)$ arbitrarily for $n \notin S_0 \union S_1 \union S_2
  \union (\bigunion \{S_{3,\rho}\}_{\rho \in G'})$.

  We will now indicate how to finish the construction to
  show both parts (1) and (2). 

  To prove (1), suppose $\mu$ is a Borel probability measure on $X$.   
  By
  Lemma~\ref{measure_diagonalization}, for each $\rho \in G'$ and $n \in
  S_{3,\rho}$, let $A_\rho$ be a $\mu$-conull Borel set and let $c_{\rho,n} \from
  A_\rho \to 2$ be a function so that for every $y \in A_\rho$, there is some $n \in
  S_{3,\rho}$ so that $g_{\rho,n}(y) = y$, $g_{\rho,n}(y)$ is undefined, or 
  $c_{\rho,n}(y) \neq c_{\rho,n}(g_{\rho,n}(y))$. We finish our definition
  of $f$ by letting 
  \[f(y)(n) = c_{\rho,n}(y) \text{ for all $\rho \in G'$ and $n \in
  S_{3,\rho}$}\]
  We claim that 
  $\hat{f} \restriction \biginters_{\rho \in G'} A_\rho$ a reduction.
  Above, we have already shown that if $\rho \notin G'$, then $\rho \cdot
  \hat{f}(y)$ is not in the range of $\hat{f}$ for all $y \in X$. So
  suppose
  $\rho \in G'$ and $y \in A_\rho$. Then if 
  $g_{\rho}^{-1}(y)$ is defined and $g_{\rho}^{-1}(y)
  \mathrel{\cancel{E_\infty}} y$, this implies that $g_{\rho,n}(y) \neq y$
  for all $n$, hence there is some $n \in S_{3,\rho}$ such that
  $c_{\rho,n}(y) \neq c_{\rho,n}(g_{\rho,n}(y))$. Hence, $\rho^{-1} \cdot
  \hat{f}(y)$ is not in the range of $\hat{f}$. 

  To prove (2), each injective partial function $g_{\rho,n}$ generates a
  Borel graph on $X$ of degree at most $2$. Let $c_{\rho,n} \from X \to 3$
  be a Borel $3$-coloring of this graph by \cite[Proposition
  4.6]{MR1667145}, and define 
  \[f(y)(n) = c_{\rho,n}(y) \text{ for all $\rho \in G'$ and $n \in
  S_{3,\rho}$}.\]
\end{proof}

Recall that given $x,y \in \cantor$, we say that $x$ is \define{many-one
reducible to $y$}, noted $x \leq_m y$, if there is a computable function $r
\from \omega \to \omega$ such that $x = r^{-1}(y)$. The associated
symmetrization of this reducibility is many-one equivalence, and is noted
$\equiv_m$. Many-one equivalence is Borel reducible to recursive
isomorphism via the function $x \mapsto \join_{i \in \omega} x$; the
function mapping $x$ to the computable join of $\omega$ many copies of $x$.
Indeed, for all reals $x$ and $y$, we have that $\join_{i \in \omega}x $
and $\join_{i \in \omega} y$ are recursively isomorphic if and only if they
are many-one equivalent. Note that many-one equivalence is closed under
countable uniform joins.

\begin{cor}\label{universality_of_ri}
\mbox{ }
\begin{enumerate}
\item The equivalence relation of recursive isomorphism on $\cantor$ is measure universal.
\item The equivalence relation of recursive isomorphism on $3^\omega$ is
universal.
\end{enumerate}
Further, the same two facts are true for many-one equivalence in place of
recursive isomorphism.
\end{cor}
\begin{proof}
  Take a computable bijection between $\F_2 \times \omega$ with $\omega$ so
  we can identify $2^{\omega}$ with $2^{\F_2 \times \omega}$, and similarly
  for the base space $3$. Then the group of computable permutations of
  $\omega$ includes the permutations required in the assumptions of
  Theorem~\ref{group_ri}.

  The fact that these facts are also true for many-one equivalence follows
  from the above observation that we may change the group $\F_2$ in the
  proof of Theorem~\ref{group_ri} into the group $\F_3$, and the proof
  will work unchanged to construct a Borel reduction from the equivalence
  relation generated by any Borel action of $\F_3$ on $X$. 
  Now take an action of $\F_3 = \<a,b,c\>$ where the first two generators $a,b$
  generate a universal countable Borel equivalence relation, and the last
  generator $c$ acts trivially so that $c \cdot x = x$ for every $x \in X$.
  Then the range of the resulting $\hat{f}$ will have
  $\hat{f}(x)((c^{k}\delta,n))$ for every $k \in \Z$. Thus, every bit in
  $\hat{f}(x)$ is duplicated infinitely many times, and so 
  $\hat{f}(x)$ and $\hat{f}(y)$ are many-one equivalent iff they
  are recursively isomorphic. 
\end{proof}

Now recursive isomorphism on $\cantor$ is not closed under countable
uniform joins, and hence Theorem~\ref{thm:closed_under_countable_joins}
does not apply to it. However,
because of the close connection between recursive isomorphism and many-one
equivalence, and the fact that all known approaches to the universality of
recursive isomorphism also give the universality of many-one equivalence,
we make the following conjecture:

\begin{conj}\label{conj:ri_not_universal}
  Recursive isomorphism on $\cantor$ is not a universal countable Borel
  equivalence relation.
\end{conj}

We can now prove part (2) of Theorem~\ref{uu_properties}.

\begin{thm}\label{measure_E_infty}
Given any uniformly universal countable Borel equivalence relation
$E^X_{\ph}$ on a standard probability space $(X,\mu)$, there is an
invariant Borel set $A$ of full measure such that $E^X_{\ph} \restrict A$
is not uniformly universal.
\end{thm}
\begin{proof}
  There is a uniform embedding of $E^X_{\ph} \restrict A$ into many-one
  equivalence for some $\mu$-conull Borel set $A$ by
  Corollary~\ref{universality_of_ri}. However, since many-one equivalence
  is not uniformly universal by
  Theorem~\ref{thm:closed_under_countable_joins}, $E^X_{\ph} \restrict A$
  can not be uniformly universal. 
\end{proof}

Our proofs above give the following interesting consequence about treeable
equivalence relations.

\begin{thm}\label{treeable_measure}
  Let $E_{\infty T} = F(\F_2,2)$ be the universal treeable countable Borel
  equivalence relation, and let $\mu$ be a Borel probability measure on the
  free part of $2^{\F_2}$. Then there exists a $\mu$-conull Borel set $A$ such
  that $E_{\infty \T}$ does not uniformly reduce to $E_{\infty \T}
  \restrict A$.
\end{thm}
\begin{proof}
  In the proof of Theorem~\ref{thm:closed_under_countable_joins}, we see that $E_{\infty \T}$ does not uniformly reduce
  to many-one equivalence on $\cantor$. However, we have shown that
  many-one equivalence is measure universal, and hence $E_{\infty \T}
  \restrict A$
  does uniformly reduce to recursive isomorphism on $\cantor$ for some
  conull set $A$ with respect to every Borel probability measure. Hence, we
  see that $E_{\infty \T}$ cannot uniformly reduce to $E_{\infty \T}
  \restrict A$.
\end{proof}

Hence, it seems reasonable to expect that the Borel cardinality of
$E_{\infty T}$ decreases after we discard a sufficiently complicated
nullset, with respect to any Borel probability measure. At the very least,
any proof that this is not the case must use a complicated
nonuniform construction.

In more recent work joint with Jay Williams, we have generalized part (2)
of Theorem~\ref{universality_of_ri} as follows.

\begin{thm}[joint with Jay Williams]\label{permutation_action_characterization}
  Suppose $X$ is a Polish space with $|X| \geq 3$, and 
  $G$ is any countable subgroup
  of $S_\infty$. Then the permutation action of $G$ on $X^\omega$
  generates a uniformly universal countable Borel equivalence relation if
  and only if there exists some $k \in \omega$ and a subgroup $H \leq
  G$ isomorphic to
  $\F_2$ such that the map $H \to \omega$ given by $\rho \mapsto \rho(k)$ is
  injective. 
\end{thm}

Essentially, not only must $\Gamma$ contain a copy of $\F_2$, as given in
(1) of Theorem~\ref{uu_properties}, this must be
witnessed in a single orbit.

\begin{proof}
  For the forward direction,
  assume that for all $n \in \omega$ and $g_0,g_1 \in G$, that there exists
  some nontrivial reduced word $h$ in $g_0$ and $g_1$ such that $h(n) = n$. Now let $f$ be
  a uniform Borel homomorphism from $E(\F_2,\cantor)$ to $E^{X^\omega}_G$ with
  uniformity function $u \from \F_2 \to G$. Since $\F_2$ is a free group, we may assume
  that $u$ is a group homomorphism. We claim that $f$ is constant on a set
  of Lebesgue measure $1$. It is enough to show that for each $n
  \in \omega$, $f(x)(n)$ is constant on a set of Lebesgue measure $1$. 

  Let $k \in \omega$ be given. Let $\F_2 = \<\alpha,\beta\>$ and consider words
  in $u(\alpha)$ and $u(\beta)$. By assumption, there must be some
  nonidentity $\gamma \in \F_2$ such that $u(\gamma)(k) = k$. We see that
  $f(\gamma \cdot x)(k) = u(\gamma) \cdot f(x)(k) = f(x)(u(\gamma)^{-1}(k))
  = f(x)(k)$. Hence, the value assigned to $x$ by $x \mapsto f(x)(k)$ is
  invariant under the map $x \mapsto \gamma \cdot x$. However, since $x
  \mapsto \gamma \cdot x$ is an ergodic transformation, the map $x \mapsto
  f(x)(n)$ must therefore be constant a.e. (Note that here we have not used
  the fact that the cardinality of $X$ is $\geq 3$.)

  The reverse direction follows from Lemma~\ref{group_ri}.
  By assumption there is a subgroup $H \leq G$
  isomorphic to $\F_2$ and a $k \in \omega$ so that the map $H \mapsto
  \omega$ defined by $h \mapsto h(k)$ is injective. Let $H' \leq H$ be a
  subgroup isomorphic to $\F_3$, and let $\alpha, \beta, \xi \in H'$ be
  generators of this subgroup. Let $\phi \from \F_2 \to H'$ be an
  isomorphism sending the two generators of $\F_2$ to $\alpha$ and $\beta$,
  and let $\pi \from \F_2 \times \omega \to \omega$ be the injection
  $\pi(\gamma,i) = \phi(\gamma) \circ \xi^i(k)$, so we can view $\omega$ as
  $\ran(\pi) \disjointunion (\omega \setminus \ran(\pi))$, and canonically
  identify it with $\F_2 \times \omega \disjointunion Z$ for $Z = \omega
  \setminus \ran(\pi)$. Viewed this, way $G$ therefore contains the
  required permutations in order to apply Lemma~\ref{group_ri} to show that the
  permutation action of $G$ on $3^\omega$ is universal.
\end{proof}

Note that one implication of this theorem is that 
if $G \leq H \leq S_\infty$ are countable and the permutation action of $G$
on $X^\omega$ generates a uniformly universal countable Borel equivalence
relation, then so does the permutation action of $H$.

With this theorem, we have finished proving all the abstract properties and
classifications of uniform universality that we have discussed in the
introduction.

\pagebreak

\section{Ultrafilters on quotient spaces}
\label{sec:ultrafilters}
\subsection{An introduction to ultrafilters on quotient spaces}

Largeness notions for the subsets of a standard Borel space play 
a central role in descriptive set theory. They are an indispensable tool
whenever we want to prove an impossibility result such as $E \nleq_B F$ for
some Borel equivalence relations $E$ and $F$.
In such situations, we must analyze all possible
Borel functions that could be counterexamples to these statements. While we
have no hope of understanding the behavior of arbitrary Borel functions
everywhere, if we are willing to discard some ``small'' part of the domain
of the function, we may gain a much clearer understanding of its behavior
on the remaining ``large'' complement. 

Formally, a notion of largeness is often taken to mean a $\sigma$-complete
filter, whose corresponding notion of smallness is its dual $\sigma$-ideal.
For example, Baire category and Borel probability measures yield such
notions of largeness: the comeager sets and conull sets. Our main goal in
this section will be to define largeness notions that
are particularly well suited for studying Borel equivalence relations
(though one could equally well study other types 
of Borel objects, such as Borel graphs, or Borel quasiorders). 
They will have much stronger properties than 
merely being $\sigma$-complete filters. 

As a motivating example, consider $E_0$: the equivalence relation of
equality mod finite on infinite binary sequences in $\cantor$. Both Lebesgue measure and Baire category have two properties
that are particularly helpful when they are used to analyze homomorphisms
from $E_0$ to other Borel equivalence relations. 
First, every
$E_0$-invariant Borel set $A \subset \cantor$ has either Lebesgue measure $0$ or
$1$, and likewise is either meager or comeager. Hence, with respect to
either of these two notions, every $E_0$-invariant set is either large or
small. Phrased another way, these largeness notions are ultrafilters when
restricted to $E_0$-invariant sets.
Second, since the restriction of $E_0$ to a comeager or conull
set cannot be smooth by standard ergodicity arguments, it must
therefore be Borel bireducible with $E_0$ by the Glimm-Effros
dichotomy~\cite{MR1057041}. 

We isolate these two phenomena into the
following definitions. Recall that given any $\sigma$-algebra $\Sigma$, an
ultrafilter on $\Sigma$ is a maximal filter of 
$\Sigma$. An ultrafilter $U$ is said to be
\define{$\sigma$-complete} if the intersection of countably many elements
of $U$ is in $U$. 

\begin{defn}
  Suppose $E$ is a countable Borel equivalence relation on a standard Borel
  space $X$ with the associated $\sigma$-algebra $\Sigma$ of $E$-invariant
  Borel sets. A
  \define{$\sigma$-complete ultrafilter on the Borel subsets of $X/E$} is
  a $\sigma$-complete ultrafilter $U$ on $\Sigma$. $U$ is said
  to \define{preserve the Borel cardinality of $E$} if for all $A \in U$,
  we have $E \sim_B E \restrict A$.
\end{defn}

Let us expand our horizons briefly and discuss this situation under the
axiom of determinacy. Under $\AD$, every ultrafilter
must be $\sigma$-complete (see for instance \cite{MR1321144}*{Proposition
28.1}) and so there can not be any nonprincipal
ultrafilter on any Polish space. Nevertheless, many ultrafilters exist on
quotients of Polish spaces by equivalence relations. For example, 
the two examples of ultrafilters we have given above as well as
all the examples we will give below will easily extend to define
ultrafilters on all subsets of the quotient $X/E$ under $\AD$. Hence, the
reader may assume we are working in this context if they prefer, and work
with genuine ultrafilters for the $\sigma$-algebra of all subsets of
$X/E$, instead of just the Borel ones.

We mention here that there are several open questions concerning
the structure of the ultrafilters on the quotient space of a
Borel equivalence relation $E$. Abstractly, we would like to know if there
is a way of classifying such ultrafilters and hence understanding exactly
what types of tools exist with which we can analyze $E$ using ergodicity
arguments. The most natural way of organizing these ultrafilters is
by Rudin-Keisler reducibility. Several important
open questions can be rephrased in this framework. For example, the question
of whether Martin's ultrafilter is $E_0$-ergodic~\cite{MR2563815} is
equivalent to asking whether there is a nonprincipal ultrafilter on the
quotient of $E_0$ that is Rudin-Keisler reducible to Martin's ultrafilter
on the Turing degrees. Along these lines, Zapletal has also asked 
whether the ultrafilters arising from measure and category are a basis for
all the ultrafilters on the quotient of $E_0$.

As described in the introduction, our main goal in this section is the
construction of new ultrafilters on the quotient space of countable Borel
equivalence relations that do not arise from measure or category. We are
motivated to construct ultrafilters in particular because of the central
role that ergodicity plays in the subject, and because of connections with
Martin's ultrafilter on the Turing degrees. Further, our desire for
ultrafilters which preserve Borel cardinality comes from our ultimate goal
of proving sharper theorems than those which are currently possible using
category (which suffers from generic
hyperfiniteness~\cite{MR2095154}*{Theorem 12.1}), and measure (which we
conjecture does not preserve Borel cardinality in many natural cases such
as for $E_{\infty T}$ and $E_\infty$: see Theorems~\ref{measure_E_infty}
and \ref{treeable_measure}).

The search for tools beyond measure theory to use in the study of countable
Borel equivalence relations has long been a central question of the
subject. Indeed, up to the present, in every case where we can prove that
there is no definable reduction from $E$ to $F$ for countable Borel
equivalence relations $E$ and $F$, we have
been able to give a measure-theoretic proof of this fact. 
We may make
this question precise under the axiom $\AD^+$, a slight technical
strengthening of $\AD$ due to Woodin which implies uniformization for
subsets of $\R^2$
with countable sections~\cite{MR2723878}
\footnote{
The Borel version of Question~\ref{reducible_vs_measure_reducible} is false
because of the known difference between $\sigma$-$\Sigma^1_1$ reducibility and Borel reducibility 
for countable Borel equivalence relations~\cite{MR1775739}*{Theorem 5.5}}.
Assume $\AD^+$, and suppose $E$ and $F$ are countable equivalence relations on the
standard Borel space $X$ and $Y$. Then say $E$ is reducible to $F$ if there
is a function $f \from X \to Y$ such that $x E y \iff f(x) F f(y)$. Say that $E$
is measure reducible to $F$ if for every probability measure $\mu$ on $X$,
there exists a $\mu$-conull set $A$ such that $E \restrict A$ is reducible
to $F$. It is open whether these two notions are distinct. That is, whether
every instance of nonreducibility arises for measure-theoretic reasons.

\begin{question}[$\AD^+$]\label{reducible_vs_measure_reducible}
If $E$ and $F$ are countable equivalence relations on standard Borel spaces and $E$ is
measure reducible to $F$, then must $E$ be reducible to $F$?
\end{question}

An affirmative answer to Question~\ref{reducible_vs_measure_reducible}
seems unlikely. For example, we have shown above that 
$E_\infty$ is
measure reducible to the equivalence relation of recursive isomorphism on
$\cantor$ and also conjectured that recursive isomorphism on $\cantor$ is not
universal (see Theorem~\ref{universality_of_ri} and
Conjecture~\ref{conj:ri_not_universal}).

We return now to the realm of Borel sets. The existence of a
$\sigma$-complete ultrafilter on the Borel subsets of $X/E$ already yields
some interesting information about $E$, related to its indivisibility,
which we proceed to define:

\begin{defn}
  Suppose that $E$ and $F$ are Borel equivalence relations. Say
  that $E$ is \define{$F$-indivisible} if for every Borel homomorphism $f$
  from $E$ to $F$, there exists an $F$ equivalence class $C$ such that $E
  \restrict f^{-1}(C)$ is Borel bireducible with $E$. In the case where $F$ is the relation of
  equality on a standard Borel space of cardinality $\kappa$, then we say
  that $E$ is $\kappa$-indivisible.
\end{defn}

Given Borel equivalence relations $E$ and $F$, $E$ being
$F$-indivisible is perhaps the antithesis of having $E$ Borel reducible to
$F$. Not only is $E \nleq_B F$, but any homomorphism $f$ from $E$ to $F$ makes
no progress whatsoever in completely classifying $E$, since there is a
single $F$-class whose preimage under $f$ has the same Borel cardinality as $E$. 

We know many examples of
interesting pairs of uncountable Borel equivalence relations $E$ and $F$
for which $E$ is $F$-indivisible (see for example the book of
Kanovei, Sabok, and Zapletal~\cite{CanonRamsey}). However, no countable Borel equivalence relations 
beyond $E_0$ were known even to be indivisible into two pieces until we
showed this for $E_\infty$ in~\cite{1109.1875}. In contrast, we note that 
Thomas~\cite{MR1903855} has shown the existence of a countable
Borel equivalence relation $E$ on an uncountable standard Borel space that
is not $2$-indivisible. Indeed, there is a countable Borel equivalence
relation $E$ where $E <_B E \oplus E$, where $E \oplus E$ is the direct sum
of two disjoint copies of $E$.

Our study of ultrafilters is connected to indivisibility in the following
way:

\begin{prop}\label{indiv_prop}
  Suppose there exists a $\sigma$-complete Borel cardinality preserving
  ultrafilter $U$ on the Borel $E$-invariant sets. Then $E$ is
  $2^{\aleph_0}$-indivisible.
\end{prop}

\begin{proof}
  This is a standard ergodicity argument.
  Let $f$ be a Borel homomorphism from $E$ to $\Delta(\cantor)$. For each
  $n \in \omega$, let $A_{n,0} = \{x : f(x)(n) = 0\}$ and $A_{n,1} = \{x :
  f(x)(n) = 1\}$. Then for each $n$, either $A_{n,0} \in U$ or $A_{n,1} \in
  U$. Hence, there is an $x \in \cantor$ such that $A_{n,x(n)} \in U$ for
  all $n$ and so $A = f^{-1}(x) = \inters_{n} A_{n,x(n)}$ has $A \in U$, so
  $E \restrict A \sim_B E$.
\end{proof}

In~\cite{1109.1875} we pointed out that for arithmetic equivalence,
the invariant sets containing an arithmetic cone are a $\sigma$-complete
ultrafilter on the $\equiv_A$-invariant sets. Since Slaman and Steel have
shown that arithmetic equivalence is a universal countable Borel
equivalence relation, and their proof relativizes, we noted that this
implies that this ultrafilter preserves Borel cardinality and hence that
$E_\infty$ is $2^{\aleph_0}$-indivisible. This answered questions of
Jackson, Kechris and Louveau, \cite{MR1900547} and Thomas 
\cite{MR2500091}*{question 3.20}, who had asked whether $E_\infty$ is
$2$-indivisible and $2^{\aleph_0}$-indivisible respectively.
In this section, we will give a new proof of this result by constructing a
large family of cardinality preserving ultrafilters on the quotient of
countable Borel equivalence relations of which $E_\infty$ is merely one
example. 

It
remains an open question to classify exactly which countable Borel
equivalence relation $E$ are such that $E_\infty$ is $E$-indivisible.
However, Martin's conjecture provides an answer to this question, as
discussed in~\cite{1109.1875}.

\subsection{A natural ultrafilter on the quotient of $E(\F_\infty,\cantor)$}
\label{subsec:ultrafilter_for_E_infty}

As in Section~\ref{sec:main_idea}, throughout this section we will have 
$I \leq \omega$ and $\{\Gamma_i\}_{i \in I}$ a collection of disjoint countable
groups. For each $i \in I$, we also fix a listing $\gamma_{i,0}, \gamma_{i,1}
\ldots$ of the nonidentity elements of $\Gamma_i$. 

We will begin by proving a strengthening of Lemma~\ref{equivariant_hom} to
show that $f$ may be chosen to be a Borel reduction. 
This will require us to generalize slightly the main game from
Section~\ref{sec:main_idea}.

\begin{defn}[The general game]\label{general_game_defn}
  Fix a strictly increasing sequence of natural numbers $(n_k)_{k \in
  \omega}$. We define the general game $G^A_i((n_k))$ for
  producing $y \in \left( \prod_{i} \cantor \right)^{\fpGammaiI}$
  identically to the game $G^A_i$, except with the following modification:
  the appropriate player will define $y(\alpha)(i)(n)$ on turn $k$ of the
  game if $t(\alpha) \geq k$, and $n_{(k-t(\alpha))} \leq \<i,n\> < n_{(k+1 -
  t(\alpha))}$. 
\end{defn}

Hence, the main game $G^A_i$ from Definition~\ref{game_defn} corresponds
to the general
game $G^A_i((n_k))$ for the sequence $n_k = k$. It is easy
to check that for all the lemmas in Sections~\ref{subsec:games_intro} and
\ref{subsec:comb_F_omega} that use game $G^A_i$, we can use the
more general game $G^A_i((n_k))$ instead.
In all the relevant inductions where we assume the turns
$<k$ have been played, one simply replaces the assumption that
$y(\alpha)(i)(n)$ has been defined for $t(\alpha) + \<i,n\> < k$ with the
assumption that $y(\alpha)(i)(n)$ has been defined where $t(\alpha) < k$ and $\<i,n\> <
n_{k - t(\alpha)}$. 

We have one more technical definition giving a growth criterion for
sequences $(n_k)_{k \in \omega}$. 

\begin{defn}
  Let $S_k = \{\beta \in \Gamma_j : t(\beta) \leq k\}$. Let $N_{\alpha,i,k}
  = \{n \in \omega \colon n_{k-t(\alpha^{-1})} \leq \<i,n\> \leq n_{k + 1 -
  t(\alpha^{-1})}\}$, so that the $n$th bit of $(\alpha \cdot y)(1)(i)$
  (which is equal to the $n$th bit of $y(\alpha^{-1})(i)$) is determined on
  turn $k$ if and only if $n \in N_{\alpha,i,k}$. Say that a sequence $(n_k)_{k \in
  \omega}$ is \define{good} if for every $i \in I$, $d \in \omega$,
  and $\alpha \in \fpGammaiI$, there are infinitely many $k$ such that
  $d^{|S_k|} < 2^{|N_{a,i,k}|}$.
\end{defn}

\begin{lemma}\label{equivariant_emb}
  Assume that $\{A_i\}_{i \in I}$ is a Borel partition of $(\prod_i
  \cantor)^{\fpGammaiI}$ and that player II has a winning strategy in the game $G^{A_j}_j((n_k))$ for a good
  sequence $(n_k)_{k \in \omega}$. 
  Then there exists some $j \in I$ and an
  injective continuous $\Gamma_j$-equivariant function $f \from
  (\cantor)^{\Gamma_j} \to A_j$, but with the additional property that $f$
  is an embedding of $E(\Gamma_j,\cantor)$ into $E(\fpGammaiI,\prod_i
  \cantor)$. 
\end{lemma} 
\begin{proof}
  Fix such a winning strategy. The $f$ we produce will be a slight variation on the one
  produced in Lemma~\ref{equivariant_hom}. Given any $g \from
  (\cantor)^{\Gamma_j} \to \cantor$, let $f_g$ be the unique equivariant
  Borel function such that for all $x \in \cantor$ we have
  $f_g(x)(\gamma)(j) = g(\gamma^{-1} \cdot x)$, and such that $f_g(x)$ is a
  winning outcome of player II's winning strategy in the game
  $G^{A_j}_j((n_k)_{k \in \omega})$. This $f_g$ exists by an analogous
  argument to that in the proof of Lemma~\ref{equivariant_hom}. In this
  notation, the $f$ produced in the proof of Lemma~\ref{equivariant_hom} is
  the function $f_{g}$ where $g(x) = x(1)$. We claim that if $g$ is a
  sufficiently generic continuous function, then $f_g$ will be as desired. 
  
  To simplify notation, identify $(\cantor)^{\Gamma_j}$ with $2^{\Gamma_j
  \times \omega}$. Fix some enumeration of $\Gamma_j \times \omega$, and
  let $2^{< \Gamma_j \times \omega}$ be the set of all functions from some
  finite initial segment of $\Gamma_j \times \omega$ to $2$. Hence, if $s,t
  \in 2^{< \Gamma_j \times \omega}$ are compatible, then $s \subset t$ or
  $t \subset s$. In general, if $s, t \in 2^{\Gamma_j \times \omega}$ and
  $\beta \in \Gamma_j$, we will say that \define{$\beta \cdot s$ is
  compatible with $t$} if for all $(\gamma,n) \in \Gamma_j \times \omega$
  so that both $s((\beta^{-1} \gamma, n))$ and $t((\gamma,n))$ are defined,
  we have $s((\beta^{-1} \gamma, n)) = t((\gamma,n))$. Otherwise, we will
  say that \define{$\beta \cdot r$ is incompatible with $s$}.
  
  The sense of genericity that we mean for continuous functions from
  $(\cantor)^{\Gamma_j} \to \cantor$ will be the usual one. Our set $\P$ of
  conditions will be the set of all finite partial functions $p$ from $2^{<
  \Gamma_j \times \omega} \to 2^{< \omega}$ such that if $s, t \in \dom(p)$
  and $s \subset t$, then $p(s) \subset p(t)$. It is clear that a generic
  filter for this poset will yield a unique continuous function $g \from
  (\cantor)^{\Gamma_j} \to \cantor$ in the usual way. 

  Note that a sufficiently generic function $g \from (\cantor)^{\Gamma_j}
  \to \cantor$ will have the property that there exists a function $\rho_j
  \from \omega \to \omega$ such that for all $x \in
  (\cantor)^{\Gamma_j}$, $\gamma \in \Gamma_j$ and $n \in \omega$, we have
  $x(1)(n) = g(x)(\rho_j(n))$. Hence, $g$
  will be injective, and thus so will $f_g$. 

  We must show that whenever $x, y \in
  (\cantor)^{\Gamma_j}$ are not $E(\Gamma_j,\cantor)$-related, and $\alpha
  \in \fpGammaiI$ is not an element of $\Gamma_j$, then 
  $\alpha \cdot f_g(x) \neq f_g(y)$. Of course, if $\alpha \in \Gamma_j$,
  then $\alpha \cdot f_g(x) = f_g(\alpha \cdot x)$, since $f_g$ is
  $\Gamma_j$-equivariant. 
  
  Fix $\alpha \in \fpGammaiI$. It suffices to prove the following. 
  Suppose $r,s \in 2^{< \Gamma_j \times \omega}$ are such
  that $\beta \cdot r$ is incompatible with $s$ for all $\beta \in
  S_{t(\alpha)}$. Then given any
  condition $p \in P$, it is dense to extend $p$ to a condition $p^*$ such
  that for all continuous functions $g$ extending $p^*$, we have that
  $\alpha \cdot f_g(x) \neq f_g(y)$ for all $x \in (\cantor)^{\Gamma_j}$
  extending $r$ and $y \in (\cantor)^{\Gamma_j}$ extending $s$. 

  We can find such a $p^*$ in the following way. Let $\sigma$ be a finite initial
  segment of $\Gamma_j \times \omega$ (according to our fixed enumeration)
  that contains the three sets 
  $\{(\beta^{-1} \gamma, n) : (\gamma, n) \in \dom(r)
  \land \beta \in S_{t(\alpha)}\}$, $\dom(s)$, and $\bigunion \{\dom(t) : t \in
  \dom(p)\}$. Let $k$ be sufficiently large so that 
  \[\min(N_{\alpha,j,k}) >  \sup(\{|p(t)| \colon t \in \dom(t)\}) \text{
  and }
  |\dom(p)|^{S_k} < 2^{|N_{\alpha,j,k}|}\]
  Let $l = \max(N_{1,j,k})$. Let $q$ be an
  extension of $p$ such that 
  \[\dom(q) = \{t \in 2^{< \Gamma_j \times \omega} : \dom(t) \subset \sigma\}\] 
  and for every $t \in \dom(q)$ with
  $\dom(t) = \sigma$, $q(t)$ has length $l$, and consists of a string in
  $\ran(p)$ followed by finitely many zeroes. Note that if $g \from
  (\cantor)^{\Gamma_j} \to \cantor$ is a
  continuous function extending $q$, then for every $x \in
  (\cantor)^{\Gamma_j}$, we have that $g(x) \restriction l$ must be
  an element of $\ran(p)$ followed by finitely many zeroes.
  Our desired $p^*$
  will be equal to some $q$ except on extensions of $s$.

  Now if $n \in N_{t(\alpha^{-1}),j,k}$, then the
  $n$th bit of $(\alpha \cdot f_g(x))(1)(j)$ will have been defined
  in the game $G^{A_j}_j((n_k))$ by player II on turn $k$,
  and hence will depend only on the values of $g(\beta \cdot x)(n')$ for $\beta \in
  \Gamma_j$ with $t(\beta) \leq k$, and $n'$ such that 
  \begin{enumerate}
    \item $\<j,n'\> < n_{k+1}$ if $t(\beta) \leq t(\alpha)$.
    \item $\<j,n'\> < n_{k - t(\alpha^{-1})}$ if $t(\beta) > t(\alpha)$.
  \end{enumerate}

  Suppose $g$ is a generic continuous function extending $q$, and $x$
  extends $r$. Then for any $\beta \in \Gamma_j$,
  \begin{enumerate}
    \item if $t(\beta) \leq t(\alpha)$, then since $\beta \cdot x$ is
    incompatible with $x$, $g(x) \restriction \{n' : \<j,n'\> < n_{k+1}\}$ must be an
    element of $\dom(p)$ followed by finitely many $0$s.
    \item if $t(\beta) > t(\alpha)$, then $g(x) \restriction \{n' :
    \<j,n'\> < n_{k - t(\alpha^{-1})}\}$ must also be an element of
    $\dom(p)$ followed by finitely many $0$s.
  \end{enumerate}
  Hence, there are at most $w = |\ran(p)|^{|S_k|}$ ways to associate an element
  in $\ran(p)$ to each $\beta \in S_k$, and so at most $|\ran(p)|^{|S_k|}$
  possible values that 
  $\alpha \cdot f_g(x)(i)(j) \restriction N_{\alpha,j,k}$ could take. Let
  $v_0, \ldots, v_{w-1} \from N_{\alpha,j,k} \to 2$ enumerate these
  possibilities. 
  Let $p^* \subset p$ be a condition with $\dom(p^*) = \dom(q)$ and such that for
  every $t$ with $\dom(t) = \sigma$, if $t$ is incompatible with $s$, then
  $p^*(t) = q(t)$. Otherwise, if $t$ extends $s$, then $p^*(t)$ is
  incompatible with $v_0, \ldots, v_{w-1}$.
\end{proof}

Theorems~\ref{g1} and \ref{g2} from the
introduction are now easy corollaries of Lemma~\ref{equivariant_emb}. Note
that as we discussed at the beginning of the introduction, $\prod_i
\cantor$ and $\cantor$ are homeomorphic, so it is fine to use the space 
$(\prod_i \cantor)^{\fpGammaiI}$ instead of $(\cantor)^{\fpGammaiI}$.

\begin{proof}[Proof of Theorem~\ref{g1}]
  Let $(n_k)_{k \in \omega}$ be a good
  sequence, and $\{A_i\}_{i \in I}$ be a Borel partition of $(\prod_i
  \cantor)^{\fpGammaiI}$. By Lemma~\ref{player_I_lemma} but for the game
  $G_i((n_k))$, there must be some $j \in I$ so that player II wins the game
  $G^{A_j}_j((n_k))$. Now apply Lemma~\ref{equivariant_emb}.
\end{proof}

The proof of Theorem~\ref{g2} simply adds the ideas from the proof of 
Lemma~\ref{free_equivariant_hom}.

\begin{proof}[Proof of Theorem~\ref{g2}]
  Let $(n_k)_{k \in \omega}$ be a good
  sequence, and $\{A_i\}_{i \in I}$ be a Borel partition of $\Free((\prod_i
  \cantor)^{\fpGammaiI})$. Now define the partitions $\{B_i\}_{i \in I}$ and
  $\{C_i\}_{i \in I}$ identically as in Lemma~\ref{free_equivariant_hom}, and
  let $A_i' = A_i \union B_i \union C_i$ so that $\{A_i'\}_{i \in I}$ is a
  Borel partition of $(\prod_i \cantor)^{\fpGammaiI}$. Player II wins 
  $G^{A_j'}_j((n_k))$ for some $j \in I$, so we can apply
  Lemma~\ref{equivariant_emb} to obtain a continuous equivariant injective
  function $f \from (\cantor)^{\Gamma_i} \to (\prod_i
  \cantor)^{\fpGammaiI}$. But then arguing identically as in the proof of
  Lemma~\ref{equivariant_emb}, $\ran(f \restriction
  \Free((\cantor)^{\Gamma_i}))$ must be contained in $A_j$.
\end{proof}

We will make a remark here on the extent to which the function
$g$ in the proof of Lemma~\ref{equivariant_emb} depends on the winning strategy for player II in the game $G_j$. 

\begin{remark}\label{inject_remark}
  We have remarked in the proof of Lemma~\ref{equivariant_emb} above that a
  sufficiently generic function $g \from (\cantor)^{\Gamma_j} \to \cantor$
  will have the property that there exists a function $\rho_j \from
  \omega \to \omega$ such that for all $x \in
  (\cantor)^{\Gamma_j}$ and $n \in \omega$,
  $x(1)(n) = g(x)(\rho_j(n))$. We remark here these functions
  $\rho_j$ can be chosen so that they will work regardless of what the
  winning strategy for player II is in the game $G^{A_j}_j((n_k))$.
  
  Precisely, we mean that if we fix a good sequence $(n_k)_{k \in \omega}$,
  then there exists a
  choice of function $\rho_i \from \omega \to \omega$ for
  each $i \in I$ such that if player II wins the game $G^{A_j}_j((n_k))$,
  then there is a continuous function $g
  \from (\cantor)^{\Gamma_j} \to \cantor$ such that $x(1)(n) =
  g(x)(\rho_j(n))$ for all $x$ and $n$, and 
  an injective continuous
  $\Gamma_j$-equivariant function $f \from (\cantor)^{\Gamma_j} \to A_j$
  that is an embedding of $E(\Gamma_j,\cantor)$ into $E(\fpGammaiI, \prod_i
  \cantor)$ and such that $f_g(x)(1)(j)
  = g(x)$ for every $x \in (\cantor)^{\Gamma_j}$.

  The reason this is true is that when we extend the condition $p$ to $p^*$
  in the proof of Lemma~\ref{equivariant_emb}, the condition $p^*$ will
  depend on the winning strategy for player II in the game $G_j$ (since $p^*$
  must diagonalize against how the strategy works), however the length of
  elements in the range of $p^*$ will always be $l$ independent of what the
  winning strategy is. Hence, if we
  choose $\rho_j$ in advance to be sufficiently fast growing, we will be
  able to alternate meeting the dense sets defined in the proof of
  Lemma~\ref{equivariant_emb} (where we extend $p$ to $p^*$) with dense
  sets ensuring that $x(1)(n) = g(x)(\rho_j(n))$.
\end{remark}

Next, we will show how we can combine Lemma~\ref{equivariant_emb} and
Lemma~\ref{player_II_lemma} to construct a 
Borel cardinality preserving ultrafilter for $E_\infty$, which will give a
new way of proving many of the theorems from Section 3
of~\cite{1109.1875}. 

\begin{lemma} \label{ultrafilter_lemma}
  Suppose $I = 2$ and $\Gamma_0$ and $\Gamma_1$ are countably infinite
  groups, and $(n_k)_{k \in \omega}$ is strictly increasing. Let
  $\mathcal{S}$ be the $\sigma$-algebra of $E(\fpGammaiI, \prod_i
  \cantor)$-invariant Borel sets. Define $U \subset \mathcal{S}$ to
  be the set of $A \in \mathcal{S}$ such that player II has a winning strategy in
  $G_0^{A}((n_k))$. Then $U$ is a $\sigma$-complete ultrafilter.
\end{lemma}
\begin{proof}
  First, it is clear that if $A \in U$ and $B \supset A$, then $B \in U$.
  Now if $B_1, B_2, \ldots \in U$ then we claim $\comp{(\inters_{m \geq 1 }
  B_m)} \notin U$. This is by contradiction, if $(\inters_{m \geq 1} B_m)
  \in U$, then by Lemma~\ref{player_II_lemma}, we could obtain a $y$ that is an outcome of
  a winning strategy for player II in $G^{(\inters_{m \geq 1}
  B_m)}_0((n_k))$ and so that
  $\gamma_{0,i} \cdot y$ is an outcome of a winning strategy for player II
  in $G^{B_m}_0((n_k))$, for each $m \geq 1$. Hence, 
  $y \in \inters_{k \geq 1} B_k$, but $\gamma_{0,k} \cdot y \in B_k$ for
  all $k \geq 1$, which is a contradiction, since the $B_k$ are invariant. 

  An analogous argument shows that if $B_1, B_2, \ldots
  \notin U$ then, $\comp{(\inters_{m \geq 1 }
  B_m)} \in U$; if player II does not have a winning strategy in
  $G_0^A((n_k))$,
  then player II has a winning strategy in $G_1^{\comp{A}}((n_k))$ by
  Lemma~\ref{player_I_lemma}. Note that here we are using the fact observed
  above that both Lemmas~\ref{player_I_lemma} and \ref{player_II_lemma} adapt to the general game
  $G_i((n_k))$.
\end{proof}

Now if $\Gamma_0 = \Gamma_1 = \F_\omega$, then
Theorem~\ref{equivariant_emb} implies that the ultrafilter from
Lemma~\ref{ultrafilter_lemma} preserves Borel cardinality.
Hence, we have a new proof
of the following theorem, by applying Proposition~\ref{indiv_prop}.

\begin{thm}[\cite{1109.1875}]
  $E_\infty$ is $2^{\aleph_0}$-indivisible. \qed
\end{thm}

In addition, Lemma~\ref{equivariant_emb} gives us a new way of proving a
couple of the other theorems from Section 3 of \cite{1109.1875}. For
example, we have the following:

\begin{thm}[\cite{1109.1875}*{Corollary 3.1}]\label{embedding_theorem}
  If $E$ is a universal countable Borel equivalence relation on a standard
  Borel space $X$, and $\{B_i\}_{i \in \omega}$ is a Borel partition of
  $X$, then there must exist a $B_i$ such that $E(\F_\omega,\cantor)
  \embeds_B E \restrict B_i$.
\end{thm}
\begin{proof} 
Let $g$ be a Borel
  reduction from $E(\fp_{i < \omega} \F_\omega, \prod_{i < \omega} \cantor)$
  to $E$. By Lusin-Novikov uniformization~\cite{MR1321597}*{Theorem 18.10}
  we can partition each 
  $g^{-1}(B_i)$ into countably many Borel sets on which $g$ is injective.
  Let $\{A_j\}_{j \in \omega}$ be the union of these countably many sets
  over every $i \in \omega$. 
  Now by Lemma~\ref{equivariant_emb},
  find a $j \in \omega$ and a Borel embedding $f$ of
  $E(\F_\omega,\cantor)$ into $E(\fp_{i < \omega} \F_\omega, \prod_{i <
  \omega} \cantor) \restrict A_j$. Then $g \circ f$ is an injective reduction
  of $E(\F_\omega, \cantor)$ to $E \restrict B_i$ for some $i$.
\end{proof}

This theorem was used in \cite{1109.1875} to also obtain the corollary that
if $E$ is a universal (under $\leq_B$) countable Borel equivalence
relation, then $F \embeds_B E$ for every countable Borel equivalence
relation $F$.

Of course, our proofs above are not so far from those of 
\cite{1109.1875}. In particular, every ultrafilter given by
Lemma~\ref{ultrafilter_lemma} refines Martin's ultrafilter on the Turing
degrees in the following way.

\begin{prop}\label{Martin_measure_refinement}
  Let $\Gamma_0, \Gamma_1$ be countably infinite computable groups, and
  $(n_k)_{k \in \omega}$ be a computable good sequence. It is
  clear that 
  $\left( \prod_i \cantor \right)^{\fpGammaiI}$ is effectively homeomorphic
  to $\cantor$, and viewed this way, the shift equivalence relation $E(\fpGammaiI,
  \prod_i \cantor)$ is a subequivalence
  relation of Turing equivalence. Now suppose that  
  $A$ is a Borel
  subset of $\left( \prod_i \cantor \right)^{\fpGammaiI}$ that is Turing
  invariant, and hence $E(\fpGammaiI, \prod_i \cantor)$-invariant. 
  Then $A$
  is in Martin's ultrafilter if and only if $A$ is in the ultrafilter $U$ given
  by Lemma~\ref{ultrafilter_lemma}.
\end{prop}
\begin{proof}
  Suppose $s$ is a winning strategy for player II in the game
  $G^A_j$.
  By Lemma~\ref{player_II_lemma}, for every $z \in \cantor$ there exists a
  $y$ that is an outcome of the strategy $s$ such that $y(1)(j) =
  z$. Hence $A$ must contain reals of arbitrarily large Turing degree and
  hence must be in Martin's ultrafilter. Conversely, if $A$ is in Martin's
  ultrafilter, then it must also be in the ultrafilter $U$ from
  Lemma~\ref{player_II_lemma}, as $\comp{A} \in U$ would imply $\comp{A}$
  is in Martin's ultrafilter.
\end{proof}

Indeed, our game in Definition~\ref{game_defn} is actually Martin's game
from~\cite{MR0227022} if $A$ is Turing invariant. The bits of $y$ are
partitioned into two computable sets: one set that player I determines, and
one set that player II determines, and then player II wins if and only if
$y$ is in $A$.

\subsection{$\K$-structurable equivalence relations}\label{subsec:k_structurable}

In this section we use the terminology and notation of
\cite{MR1900547}*{Section 2.5}. Suppose $L = \{R_i: i \in I\}$ is a
countable relational language and $X_L$ is the associate space of
$L$-structures whose universe is some set $I \leq \omega$. Let $\K \subset X_L$
be a Borel class of $L$-structures closed under isomorphism. By a theorem
of Lopez-Escobar, such $\K$ are exactly those classes of structures defined
by some $L_{\omega_1,\omega}$ sentence~\cite{MR1321597}*{Theorem 16.8}. A
countable Borel equivalence relation $E$ on a standard Borel space $X$ is
said to be \define{Borel $\K$-structurable} if there is a Borel assignment
of a $\K$-structure to each $E$-class whose universe is that class.
Precisely, we mean that there are Borel relations $\{Q_i : i \in I\}$ on
$X$ where $Q_i$ and $R_i$ have the same arity for all $i \in I$, and so
that for each $E$-class $C$, the structure whose universe is $C$ and whose
relations are $Q_i \restrict C$ is isomorphic to some structure in $\K$.

Recall that if $E$ and $F$ are Borel equivalence relations on the standard
Borel spaces $X$ and $Y$, then a Borel homomorphism $f \from X \to Y$ from
$E$ to $F$ is said to be \define{class bijective} if for every $E$-class
$C$, $f \restriction C$ is a bijection between $C$ and some $F$-class. It
is easy to see that if there is a class bijective homomorphism Borel
homomorphism from $E$ to $F$, and $F$ is Borel $\K$-structurable, then $E$
is also Borel $\K$-structurable.

Many natural classes of countable Borel equivalence relations are exactly
the $\K$-structurable equivalence relations for some such class of
structures $\K$. For example, the class $\K$ of trees yields the class of
treeable equivalence relations. Similarly, the class of hyperfinite
equivalence relations corresponds to the class of structures that are
increasing sequences of finite equivalence relations whose union is the
universe of the structure. Further, given a class $\mathcal{C}$ of
$\K$-structurable equivalence relations for some $\K$, the class of
equivalence relations that are finite index over elements of $\mathcal{C}$,
and the class of equivalence relations that are increasing unions of
elements of $\mathcal{C}$ are also both structurability classes. 

Ben Miller has pointed out the following theorem, which generalizes many
ad-hoc universality proofs in the field of countable Borel equivalence
relations, such as \cite[Proposition 7.1]{MR1667145} for graphs.
The proof we give is a simplified version of our earlier
argument suggested by Kechris. Recall that the notation $E \embeds_B^i F$
indicates that there is an injective Borel reduction from $E$ to $F$ whose
range is $F$-invariant.

\begin{thm}[Miller, after Jackson, Kechris, and
Louveau~\cite{MR1900547} and Kechris, Solecki, Todorcevic~\cite{MR1667145}]\label{universal_structurable}
  If $\K$ is a Borel class of countable relational structures closed under
  isomorphism, then there is a universal countable Borel $\K$-structurable
  equivalence relation, which we note $E_{\infty \K}$. That is, given any
  countable Borel $\K$-structurable equivalence relation $E$, then $E
  \embeds_B^i E_{\infty \K}$.
\end{thm}
\begin{proof}
  Let $E_\infty$ be an invariantly universal
  countable Borel equivalence relation on the space $Y_\infty$ so that for all countable Borel
  equivalence relations $F$, we have that 
  $F \embeds_B^i E_\infty$ (see \cite{MR1149121}). For example, let $E_\infty =
  E(\F_\omega, \cantor)$. Let $y \mapsto f_y$ be a Borel map from
  $Y_\infty \to (Y_\infty)^{\leq \omega}$ associating
  to each $y \in Y_\infty$ a bijection $f_y: I \to Y_\infty$ from some $I
  \leq \omega$ to the equivalence class $[y]_\infty$ of $y$.

  Now let $Y_{\infty \K}$ be the set of $(x,y) \in X_L \times
  Y_\infty$ such that $x \in \K$, and the universe of $x$ has the same
  cardinality as $[y]_{E_\infty}$. Now $E_{\infty \K}$ is defined to be the
  equivalence relation on $A$ where $(x_0,y_0) \mathrel{E_{\infty \K}}
  (x_1, y_1)$ iff $y_0 E_\infty y_1$ and the structure on $[y_0]_\infty$
  induced by pushing forward $x_0$ under $f_{y_0}$ is equal to the
  structure on $[y_1]_\infty$ obtained by pushing forward $x_1$ under
  $f_{y_1}$.

  Suppose $F$ is a $\K$-structurable countable Borel equivalence
  relation on $X$ and let $Q_0, Q_1, \ldots$ be Borel relations on $X$
  giving a $\K$-structuring of $F$. Let $g \from X \to Y_\infty$
  witness that $F \embeds_B^i E_\infty$. Then $F \embeds_B^i E_{\infty
  \K}$ via the function $x \mapsto (h(x),g(x))$ where $h \from X \to X_L$
  is the function sending $x$ to the structure $h(x) \in X_L$ on the set
  $\dom(f_{g(x)})$ where 
  \[R_i(n_0, \ldots, n_k) \iff Q_i(f_{g(x)}(n_0), \ldots, f_{g(x)}(n_k))\]

  It is also clear that $E_{\infty \K}$ is Borel $\K$-structurable. Define the
  relations $Q_i$ on $A$ witnessing $E_{\infty \K}$ is Borel
  $\K$-structurable by setting 
  \[Q_i((x_0,y_0), \ldots (x_n,y_n)) \iff x_0 \models R_i(f^{-1}_{y_0}(y_0), f^{-1}_{y_0}(y_1),
  \ldots, f^{-1}_{y_0}(y_n))\]
\end{proof}

We have the following trivial corollary of Sacks's theorem that Turing
cones have measure zero, which generalizes \cite[Theorem 3.10]{1109.1875}.

\begin{thm} \label{universal_on_null}
  Suppose $\K$ is a Borel class of countable structures closed under
  isomorphism, and let $E_{\infty \K}$ be the universal $\K$-structurable
  equivalence relation on the space $Y_{\infty \K}$. 
  Then if $\mu$ is a Borel probability measure on $Y_{\infty \K}$, there
  is a $\mu$-null Borel set $A$ so that $E_{\infty \K} \embeds^i_B
  E_{\infty \K} \restriction A$.
\end{thm}
\begin{proof}
  Our definition of $E_{\infty \K}$ in Theorem~\ref{universal_structurable}
  has $Y_\infty = (\cantor)^{\F_\omega}$ which is computably homeomorphic to
  $\cantor$, and $X_L$ also computably homeomorphic to $\cantor$ relative
  to the language $L$. Hence, it makes sense to discuss computability of
  elements of the set $Y_{\infty K} \subset X_L \times Y_\infty$ on which $E_{\infty \K}$ is defined. 

  By the relativized version of Sacks's theorem~\cite{MR0186554}, if 
  $x \in \cantor$ can
  compute a representation of $\mu$, and $y >_T x$, then the cone $\{z \in
  \cantor : z \geq y\}$ has $\mu$-measure $0$. Choose such a $y$, and note
  that since $E_{\infty \K}$ is induced by a Borel action of $\F_\omega$,
  there is an invariant Borel embedding $E_{\infty K}$ into $E_{\infty}$ via the
  function $g(x)(\gamma) =  y \oplus \gamma^{-1} \cdot x$. Using $g$ to
  define an invariant Borel embedding of $F = E_{\infty \K}$ into $E_{\infty \K}$ as in the proof of
  Theorem~\ref{universal_structurable}, then the range of the resulting
  function is contained in the cone
  $\{z : z \geq_T y\}$, which therefore has $\mu$-measure $0$.
\end{proof}

We mention here that there is a wealth of open problems related to how
model theoretic properties of $\K$ are related to the Borel cardinality of
$E_{\infty \K}$. For example, suppose $\K$ is the isomorphism class of a
single structure. Can we characterize when $E_{\infty \K}$ is smooth? How
about in the special case when the structure is a 
Fra\"iss\'e Limit?\footnote{These questions have been resolved by \cite{ChenKechris}}

\subsection{Ultrafilters for $\K$-structurable equivalence relations closed
under independent joins}

Our next goal is to use arguments similar to those in
Section~\ref{subsec:ultrafilter_for_E_infty}
to obtain ultrafilters for a class
of universal $\K$-structurable equivalence
relations. The key property $\K$ must have to allow our arguments to work
is that the class of $\K$-structurable equivalence relations must be closed
under independent joins. 

Say that the collection of $\K$-structurable countable Borel equivalence
relations is \define{closed under binary independent joins} if whenever $E$
and $F$ are independent countable Borel $\K$-structurable equivalence
relations on a standard Borel space $X$,
then their join $E \vee F$ is also $\K$-structurable Borel
equivalence relations. Say that the $\K$-structurable countable Borel equivalence
relations are \define{closed under countable independent joins} if whenever
$\{E_i\}_{i \in I}$ are
independent countable Borel $\K$-structurable equivalence relations on the
same standard Borel space, then their join $\bigvee_{i \in I} E_i$ is also
a $\K$-structurable Borel equivalence relation. 

Trees, contractible simplicial complexes of dimension $\leq n$, and the
class of all countable structures are all examples of classes of structures
$\K$ so that the $\K$-structurable Borel equivalence relations are closed
under countable independent joins. So are increasing unions of these
examples. Locally finite trees and Cayley graphs of groups $\Gamma$ for
which $\Gamma * \Gamma \cong \Gamma$ are examples of classes of structures
$\K$ so that the $\K$-structurable Borel equivalence relations are closed
under binary independent joins.

For the specific case of contractible simplicial complexes of dimension
$\leq n$, we note that the universal structurable equivalence relation for
this class is Borel bireducible with the equivalence relation of
isomorphism of contractible simplicial complexes of dimension $\leq n$ by
standard arguments. Gaboriau has shown~\cite{MR1953191} that these
equivalence relations form a proper hierarchy under $\leq_B$. See
also~\cite{MR2155451}.

\begin{thm}\label{K_ultrafilters}
  Suppose $\K$ is a Borel class of countable structures so that the class
  of $\K$-structurable Borel equivalence relations is closed under binary
  independent joins. 
  Then there is an ultrafilter on the
  invariant Borel sets of $E_{\infty \K}$ that preserves Borel cardinality.
  Hence, $E_{\infty \K}$ is $2^{\aleph_0}$-indivisible.
\end{thm}
\begin{proof}
  The proof of this theorem uses ideas from the proofs of 
  Lemma~\ref{free_equivariant_hom}, Remark~\ref{inject_remark}, and
  Lemma~\ref{ultrafilter_lemma}.

  Let $I = 2$, and $\Gamma_i = \F_\omega$ for $i \in I$. Let $\hat{X_i} \subset
  (\cantor)^{\Gamma_i}$ be the range of a class bijective Borel embedding
  of $E_\infty \K$ into $E(\Gamma_i, \cantor)$, so 
  $E(\Gamma_i, \cantor)\restriction \hat{X_i}$ is Borel isomorphic to $E_{\infty
  \K}$ for every $i$. 

  Fix functions $\rho_i \from \omega \to \omega$
  as in Remark~\ref{inject_remark}, and let $\pi_i' \from (\prod_i
  \cantor)^{\fpGammaiI} \to (\cantor)^{\Gamma_i}$ be the function
  \[\pi_i'(x)(\gamma)(n) = x(\gamma)(i)(\rho_i(n))\]
  so that $\pi_i'$ is $\Gamma_i$-equivariant. Note that $\pi_i'$ is
  essentially just the function $\pi_i$ from Definition~\ref{easy_proj}
  modified by $\rho_i$.

  In our proof we will be considering functions $f_g \from
  (\cantor)^{\Gamma_i} \to (\prod_i \cantor)^{\fpGammaiI}$ as defined in the
  proof of Lemma~\ref{free_equivariant_hom} and Remark~\ref{inject_remark},
  and so together, these functions will have the property that 
  \[\pi_i'(f_g(x)) = x.\]
  Let
  $X_i = \{y \in (\prod_i \cantor)^{\fpGammaiI}: \pi_i'(y) \in \hat{X_i} \land \forall \gamma \in
  \Gamma_i (\gamma \cdot y \neq y \implies \gamma \cdot \pi_i'(y) \neq
  \pi_i'(y))\}$
  so $X_i$ is $\Gamma_i$-invariant.

  Let $Y$ be the largest invariant set of $y \in Y$ such that $y \in X_i$
  for every $i \in I$, and let $\{C_i\}_{i \in I}$ be a Borel partition of the
  complement of $Y$ as in Lemma~\ref{Y_complement_lemma}. 
  Let $E_i$ be the
  equivalence relation on $Y$ where $x \mathrel{E_i} y$ if $\exists \gamma
  \in \Gamma_i (\gamma \cdot x = y)$. Let $\hat{Y}$ be the
  largest $\fpGammaiI$-invariant subset of $Y$ on which the $E_i$ are
  independent. Hence the $E_i$ are everywhere non-independent on $Y
  \setminus \hat{Y}$, so let $\{B_i\}_{i \in I}$ be a Borel partition of $Y
  \setminus \hat{Y}$ as in Lemma~\ref{partition_non_independent}.

  Now for every $i \in I$, $E_i \restriction \hat{Y}$ is $\K$-structurable,
  since $\pi_i'$ gives a class bijective homomorphism from $E_i \restriction
  \hat{Y}$ to $E(\F_\omega,\F_\omega) \restriction X$ which is
  $\K$-structurable. Hence the join $\bigvee_i (E_i \restriction \hat{Y})$
  on $\hat{Y}$ is also $\K$-structurable by assumption. Let $E = \bigvee_i
  (E_i \restriction \hat{Y})$ be the equivalence relation on $\hat{Y}$
  generated by the shift action of $\fpGammaiI$. 

  Given any Borel partition $\{A_i\}_{i \in I}$ of $\hat{Y}$, let $A_i = A_i
  \union B_i \union C_i$, so $\{A_i'\}_{i \in I}$ will be a Borel partition
  of the whole space $(\prod_i \cantor)^{\fpGammaiI}$.

  Let $(n_k)_{k \in \omega}$ be a good sequence. Suppose player II wins the
  game $G^{A_j'}_j((n_k))$ and 
  $f_g$ is the associated injective continuous
  $\Gamma_j$-equivariant function $f_g \from (\cantor)^{\Gamma_j} \to A_j'$
  constructed by Lemma~\ref{equivariant_emb} and
  Remark~\ref{inject_remark}. Then since $\pi'_j(f_g(x)) = x$ for all $x$,
  we see that $\ran(f_g \restriction \hat{X_j})$ is contained in $X_j$, and so
  $\ran(f_g \restriction \hat{X_j})$ is contained in $\hat{Y}$, since it
  cannot meet $B_j$ by Lemma~\ref{partition_non_independent} or $C_j$ by
  Lemma~\ref{Y_complement_lemma} (as in the proof of Lemma~\ref{free_equivariant_hom}). 

  Since $E(\Gamma_i,\cantor) \restriction \hat{X_j}$ is Borel isomorphic to
  $E_{\infty \K}$, we have just shown that $E_{\infty \K} \leq_B E$, and
  since $E$ is itself $\K$-structurable and so $E \leq_B E_{\infty \K}$, we
  have that $E_{\infty \K} \sim_B E$. 

  Now the collection of $E$-invariant
  Borel sets $A$ such that player II wins $G^{A \union B_0 \union C_0}_0((n_k))$
  is an ultrafilter on the $E$-invariant Borel sets by an
  identical proof to Lemma~\ref{ultrafilter_lemma}. Furthermore, if player II
  wins this game, then $E_{\infty \K} \leq_B E \restriction A$ by our
  argument above. 
\end{proof}

This constitutes part (2) of Theorem~\ref{KS_properties} from the
introduction. We finish by proving part (3) of this theorem.

\begin{thm}\label{K_embedding}
  Suppose $\K$ is a Borel class of countable structures closed under
  isomorphism, so that the class of $\K$-structurable countable Borel
  equivalence relations is closed under countable independent joins. Let
  $E_{\infty \K}$ be the universal $\K$-structurable Borel equivalence
  relation on $Y_{\infty \K}$. Then if 
  $\{A_i\}_{i \in \omega}$ is a Borel partition of $Y_{\infty \K}$ into
  countably many Borel pieces, there is some $A_i$ such that $E_{\infty \K}
  \embeds_B E_{\infty \K} \restriction A_i$. Hence, for all countable Borel
  equivalence relations $F$, $E_{\infty
  \K} \leq_B F$ implies $E_{\infty \K} \embeds_B F$. 
\end{thm}
\begin{proof}
  If we let $I = \omega$, then the proof of Theorem~\ref{K_ultrafilters}
  shows that $E = \bigvee (E_i \restriction \hat{Y})$ is
  $\K$-structurable, and hence $E \embeds_B E_{\infty \K}$, and if
  $\{A_i\}_{i \in \omega}$ is any Borel partition of $\hat{Y}$, then there
  is some $j$ such that $E_{\infty \K} \embeds_B E \restriction A_j$ (since
  there must be some $j$ so that 
  player II wins the game $G_j^{A_j \union B_j \union C_j}$). Hence, since
  $E$ and $E_{\infty \K}$ are bi-embeddable, the first half of the theorem
  follows for $E_{\infty \K}$.

  Suppose now that $F$ is a countable Borel equivalence relation on the
  standard Borel space $X$, and 
  $E_{\infty \K} \leq_B F$ via the Borel reduction $g \from Y_{\infty \K}
  \to X$. By
  Lusin-Novikov uniformization, we can partition $Y_{\infty \K}$ into 
  countably many Borel sets $\{A_i\}_{i < \omega}$ on which $g$ is
  injective. But then $E_{\infty \K} \embeds_B E_{\infty K} \restriction
  A_i$ for some $j$.
\end{proof}

\bibliography{references}

\end{document}